\documentclass[final,12pt]{colt2018} 

\newcommand{\R}{\mathbb{R}} 
\newcommand{\E}{\mathbb{E}}
\newcommand{\N}{\mathcal{N}} 
\newcommand{\HS}{\mathrm{HS}} 
\DeclareMathOperator{\Tr}{Tr}
\renewcommand{\part}[2]{\frac{\partial #1}{\partial #2}}
\newcommand{\I}{\mathrm{I}}
\DeclareMathOperator{\grad}{grad}
\DeclareMathOperator{\Hess}{Hess}
\newcommand{\T}{\mathsf{T}}
\newcommand{\M}{{M}}
\renewcommand{\P}{\mathcal{P}}
\DeclareMathOperator{\Var}{Var}
\newcommand{\bN}{\mathbb{N}}
\newcommand{\F}{\mathsf{F}}
\newcommand{\B}{\mathsf{B}}

\newtheorem{question}{Question}

\title[Sampling as optimization in the space of measures]{
Sampling as optimization in the space of measures: \\
The Langevin dynamics as a composite optimization problem}
\usepackage{times}

 \coltauthor{\Name{Andre Wibisono} \Email{aywibisono@wisc.edu}\\
 \addr Department of Electrical \& Computer Engineering \\
 University of Wisconsin -- Madison
 }

\begin{document}

\maketitle

\begin{abstract}
We study sampling as optimization in the space of measures.
We focus on gradient flow-based optimization with the Langevin dynamics as a case study.
We investigate the source of the bias of the unadjusted Langevin algorithm (ULA) in discrete time, and consider how to remove or reduce the bias.
We point out the difficulty is that the heat flow is exactly solvable, but neither its forward nor backward method is implementable in general, except for Gaussian data.
We propose the symmetrized Langevin algorithm (SLA), which should have a smaller bias than ULA, at the price of implementing a proximal gradient step in space.
We show SLA is in fact consistent for Gaussian target measure, whereas ULA is not.
We also illustrate various algorithms explicitly for Gaussian target measure with Gaussian data, including gradient descent, proximal gradient, and Forward-Backward, and show they are all consistent.
\end{abstract}

\section{Introduction}
We study sampling as optimization in the space of measures.
In this paper 
we focus on gradient flow-based optimization with the Langevin dynamics as a case study.
Our starting point is the key result of~\cite{JKO98} that the Langevin dynamics in space corresponds to the gradient flow of the relative entropy functional in the space of measures with the Wasserstein metric.
This is why running the Langevin dynamics is useful for sampling: It is the steepest descent flow that attracts any initial distribution to the stationary target measure.
Our motivating question is:
\begin{center}
{\em
Is there an {\bf implementable} discretization of the Langevin dynamics that is consistent and converges exponentially fast under the logarithmic Sobolev inequality?
}
\end{center}

Recall in general, gradient flow converges exponentially fast under a gradient-domination condition that is weaker than strong convexity.
In the space of measures with the relative entropy functional, this gradient-domination condition is the logarithmic Sobolev inequality (LSI)~\cite[]{OV00}.
Therefore, in continuous time, the Langevin dynamics converges exponentially fast under LSI.
In discrete time, the situation is less clear.
A basic discretization known as the unadjusted Langevin algorithm (ULA) is {\em biased}, which means it converges to a limit different from the target measure.
This bias exists for arbitrarily small (fixed) step size, even for a Gaussian target measure.
This led to proposals to correct the bias, such as the Metropolis-Hastings correction step~\cite[]{RT96,DCWY18} or decreasing the step size; however, the resulting algorithms become more complicated, so here we focus on ULA in order to see the basic structure.

Even with the bias, we can use ULA to get a one-time sampling algorithm as follows:
given an error threshold, first choose a small enough step size so the bias is within the threshold, then run ULA with that step size to approximate convergence.
See for example~\cite{D17,D17a,DK17,DM16,CB17}.

The bias in the Langevin algorithm is actually puzzling.
The bias is typically attributed to the fact that ULA is a discretization of the continuous-time Langevin dynamics, so it necessarily has a discretization error.
However, this attribution is misplaced, because
the discretization error is the deviation between ULA and the Langevin dynamics at small time,
while
we are concerned with
the asymptotic bias of ULA at large time. 
And indeed it is possible for a discretization algorithm to be consistent (unbiased).
For example, gradient descent is a discretization of gradient flow; under the gradient-domination condition, both gradient flow and gradient descent converge to the minimizer exponentially fast.
Since the Langevin dynamics is a gradient flow and it converges exponentially fast under LSI, we expect the gradient descent version to also converge exponentially fast under LSI, hence our motivating question.
We note that the class of measures satisfying LSI is rather large and closed under bounded perturbation, so even multimodal distributions satisfy LSI.

Why is ULA biased?
It cannot be the gradient descent discretization of the Langevin dynamics.
Rather, it is performing the {\em Forward-Flow} (FFl) discretization, which in general is biased.
Here we observe that the problem of minimizing relative entropy in the space of measures is a composite optimization problem, which means the objective function is a sum of two terms.
Indeed, relative entropy can be written as the sum of negative entropy and the expected function value, where the function is the negative log density of the target measure.
ULA (FFl) is a two-step algorithm: First, it applies the forward method to the expected function value, which is implemented by the usual gradient descent step in space.
Second, it applies the exact gradient flow of the negative entropy, the heat flow, which is implemented by adding independent Gaussian noise in space.
Therefore, ULA is {\em implementable}, which means we can run it in space with a given sample,
but it is biased.
We can compute the bias explicitly for the Ornstein-Uhlenbeck (OU) process, which is the case of Gaussian target measure; see Example~\ref{Ex:OU2} below.

For a composite optimization problem, the algorithm of choice is the {\em Forward-Backward} (FB) algorithm, which is a composition of the forward method for one term and the backward method for the other.
The FB algorithm is consistent, which crucially uses the fact that the backward method is the adjoint of the forward method.
Furthermore, FB converges at exponential rate under the gradient-domination condition and some smoothness assumptions.
For our problem of optimizing relative entropy, FB means we want to run the forward method for the expected function value (which we can do via the usual gradient descent), and run the backward method for negative entropy (which we cannot do in general).
Indeed, the main interesting difficulty is that for negative entropy, its exact gradient flow (the heat flow) is implementable via Gaussian noise, but neither its forward nor backward method is implementable.\footnote{
This is the opposite of what typically happens in optimization, where we cannot run gradient flow but we can run the gradient descent algorithm. 
}
However, in one case, namely for Gaussian initial data, we can solve the backward method for negative entropy. 
Therefore, for the OU process with Gaussian initial data, we can solve the FB algorithm and see it is indeed consistent; see Example~\ref{Ex:OUFB}.

Finally, if we cannot remove the bias, we can try to reduce it.
We propose the {\em symmetrized Forward-Flow} (SFFl) algorithm, which is a composition of FFl and its adjoint.
The SFFl algorithm is symmetric and has order 2;
therefore, its bias is also of order 2, smaller than FFl.
Applying SFFl to the Langevin dynamics yields the {\em symmetrized Langevin algorithm} (SLA), which is a composition of ULA and its adjoint.
SLA requires being able to run the backward method for the expected function value, which is implemented by the proximal gradient step for the function in space.
This may require some numerical computation in each iteration, but the prize is a sampling algorithm that should have a smaller bias.
As an example, we show how to solve SLA explicitly for the OU process.
We see in this case SLA is in fact consistent; see Example~\ref{Ex:OUSLA}.
As another example, we show how to implement SLA for a mixture of two Gaussians as target measure, by solving an explicit one-dimensional calculation in each iteration.
We show using synchronous coupling under strong log-concavity that SLA converges exponentially fast to its limiting measure.

\section{Sampling as optimization in the space of measures}
\label{Sec:Samp}

Sampling can be formulated as optimization in the space of measures.
Indeed, to sample from a target distribution, it suffices to optimize an objective function in the space of measures that is minimized at the target distribution.
Thus, we can translate methods from optimization to sampling by applying them to the optimization problem in the space of measures, provided the resulting methods are implementable as (possibly stochastic) algorithms in space.

In principle, we can choose any objective function that is minimized at the target distribution.
But in the space of measures, there is a special function that works, which is the {\em relative entropy}:
\begin{align}\label{Eq:RelEnt}
H_\nu(\rho) = \int \rho \log \frac{\rho}{\nu}.
\end{align}
This is also known as the Kullback-Leibler (KL) divergence.
Here $\rho$ is a probability measure on $\R^n$ absolutely continuous with respect to $\nu$, and the integral above is a shorthand for $\int_{\R^n} \rho(x) \log \frac{\rho(x)}{\nu(x)} dx$.
Relative entropy is nonnegative: $H_\nu(\rho) \ge 0$; and it is minimized at the target measure: $H_\nu(\rho) = 0$ if and only if $\rho = \nu$.
Furthermore, $\nu$ is the only stationary point of $H_\nu$, even when $\nu$ is multimodal.
Therefore, if we can minimize $H_\nu$, then we can sample from $\nu$.

Let $\nu = e^{-f}$, or equivalently $f = -\log \nu$.
Relative entropy decomposes as
 a sum of two terms:
\begin{align}\label{Eq:RelEnt2}
H_\nu(\rho) = \E_\rho[f] -H(\rho)
\end{align}
where $\E_\rho[f] = \int \rho f$ is the expected value of $f$
and $-H(\rho) = \int \rho \log \rho$ is the negative entropy of $\rho$.

\subsection{The Langevin dynamics as the gradient flow of relative entropy}

We wish to minimize the relative entropy functional~\eqref{Eq:RelEnt} in the space of measures.
A general strategy to minimize a function is to run the gradient flow dynamics. 
This requires a metric structure~\citep{AGS08}.
In the space of measures over $\R^n$ there is a nice choice of metric, which is the Wasserstein metric induced by the quadratic distance function~\citep{Vil03,Vil08}.

In the space of measures with the Wasserstein metric, the gradient flow of relative entropy~\eqref{Eq:RelEnt} is
the following partial differential equation, known as the {\em Fokker-Planck equation}: 
\begin{align}\label{Eq:FP}
\part{\rho}{t} \,=\, \nabla \cdot \Big( \rho \nabla \log \frac{\rho}{\nu}\Big) \,=\, \nabla \cdot (\rho \nabla f) + \Delta \rho.
\end{align}
Here $\rho = \rho(x,t)$ is a smooth positive density evolving over time. 
This is the key result of~\cite{JKO98}, which has been extended to vast generalities~\citep{Vil08}.
So if we can follow the flow of the Fokker-Planck equation in the space of measures, then we converge to the target measure $\nu = e^{-f}$.
Furthermore, if $\nu$ satisfies the logarithmic Sobolev inequality (LSI), then the convergence is exponentially fast, see~$\S\ref{Sec:LSI}$.
However, can we implement this in space?

It turns out the Fokker-Planck equation is the continuity equation 
of the {\em Langevin dynamics}, which is the following stochastic differential equation in space:
\begin{align}\label{Eq:LD}
dX = -\nabla f(X) \, dt + \sqrt{2} \, dW.
\end{align}
Here $X = (X_t)_{t \ge 0}$ is a stochastic process and $W = (W_t)_{t \ge 0}$ is the standard Brownian motion in $\R^n$. 
That is, if $X_t \sim \rho_t$ evolves following the Langevin dynamics~\eqref{Eq:LD} in space, then $\rho(x,t) = \rho_t(x)$ evolves following the Fokker-Planck equation~\eqref{Eq:FP} in the space of measures; see for example~\cite[$\S11$]{Mac92}.

This means the Fokker-Planck flow~\eqref{Eq:FP} is implementable by the Langevin dynamics~\eqref{Eq:LD}, as long as we can follow the stochastic process~\eqref{Eq:LD} exactly.
In one case, namely for the Gaussian target measure, we know the exact solution; see Example~\ref{Ex:OU} below.
In general, we need to discretize to obtain an algorithm in space;
we discuss this further below.

\begin{example}[Ornstein-Uhlenbeck.]\label{Ex:OU}
Let $\nu = \N(\mu,\Sigma)$ be the Gaussian target measure with mean $\mu \in \R^n$ and covariance $\Sigma \succ 0$, so $f(x) = \frac{1}{2}(x-\mu)^\top \Sigma^{-1} (x-\mu) + \frac{1}{2} \log \det(2\pi\Sigma)$ is quadratic.
The Langevin dynamics~\eqref{Eq:LD} has a linear drift, and is known as the {\em Ornstein-Uhlenbeck} (OU) process:
$$dX = -\Sigma^{-1}(X-\mu) \, dt + \sqrt{2} \, dW.$$
This has an exact solution as an It\^o integral.
In particular, the solution at each time $t \ge 0$ satisfies
$$X_t-\mu \stackrel{d}{=} e^{-t \Sigma^{-1}} (X_0-\mu) + \Sigma^{\frac{1}{2}} \Big(I-e^{-2t\Sigma^{-1}}\Big)^{\frac{1}{2}} Z$$
where $Z \sim \N(0,I)$ is independent of $X_0$.
Thus, $X_t \sim \rho_t$ converges to $\N(\mu,\Sigma)$ exponentially fast.
\end{example}

\subsection{The unadjusted Langevin algorithm} 
\label{Sec:ULA}

A practical discretization of the Langevin dynamics~\eqref{Eq:LD} is the {\em unadjusted Langevin algorithm} (ULA):
\begin{align}\label{Eq:ULA}
x_{k+1} = x_k - \epsilon \nabla f(x_k) + \sqrt{2\epsilon} z_k
\end{align}
where $z_k \sim \N(0,I)$ is independent of $x_k$. 
Here $\epsilon > 0$ is a step size, which is equal to the time step in the discretization.

As $\epsilon \to 0$, the ULA iteration~\eqref{Eq:ULA} converges to the Langevin dynamics~\eqref{Eq:LD}~\citep[$\S11$]{Mac92}.
However, when $\epsilon > 0$, ULA is biased, which means it does not converge to the target distribution $\nu = e^{-f}$.
This bias is present even for Gaussian target measure, see Example~\ref{Ex:OU2}.

\begin{example}[ULA for OU.]\label{Ex:OU2}
Let $\nu = \N(\mu,\Sigma)$.
ULA is $x_{k+1} - \mu = (I - \epsilon \Sigma^{-1})(x_k-\mu) + \sqrt{2\epsilon} z_k$.
Unfolding and using the fact that the sum of independent Gaussians is Gaussian, 
we can write
$$x_k - \mu \stackrel{d}{=} A_\epsilon^k(x_0-\mu) + \sqrt{2\epsilon} (I-A_\epsilon^2)^{-\frac{1}{2}} (I-A_\epsilon^{2k})^{\frac{1}{2}} z$$
where $A_\epsilon = I - \epsilon \Sigma^{-1}$ and $z \sim \N(0,I)$ is independent of $x_0$.
For $0 < \epsilon < 2\lambda_{\min}(\Sigma)$, 
$\lim\limits_{k \to \infty} A_\epsilon^k = 0$.
Therefore, $x_k \stackrel{d}{\to} \mu + \sqrt{2\epsilon} (I-A_\epsilon^2)^{-\frac{1}{2}} z$.
Thus, ULA for OU has the limit measure
$$\nu_\epsilon = \N\big(\mu,\,\Sigma(I - \tfrac{\epsilon}{2} \Sigma^{-1})^{-1}\big).$$
At $\epsilon = 0$, this is the target measure $\nu = \N(\mu,\Sigma)$.
For $\epsilon > 0$, $\nu_\epsilon \neq \nu$, so ULA is biased.
The bias is
$$W_2(\nu, \nu_\epsilon) = \|\Sigma^{\frac{1}{2}}-\Sigma^{\frac{1}{2}}(I - \tfrac{\epsilon}{2} \Sigma^{-1})^{-\frac{1}{2}}\|_{\HS} \,=\, \frac{\epsilon}{4} \sqrt{\Tr(\Sigma^{-1})} + O(\epsilon^2).$$
where $\|B\|_{\HS} = \sqrt{\Tr(B^2)}$ is the Hilbert-Schmidt norm of a symmetric matrix $B$.
\end{example}
Above, we have used the formula for the Wasserstein distance between Gaussians~\cite[]{T11}.

\subsubsection{Convergence to the biased limit}

We recall the following contraction result using synchronous coupling under strong log-concavity; see~\cite[Lemma~1]{D17a} or Appendix~\ref{App:ULASynch}.
Note the similarity with the corresponding result in strongly convex optimization, e.g.,~\cite[Theorem~2.1.15]{Nesterov04}.

\begin{lemma}\label{Lem:ULASynch}
Suppose $\nu$ is $\alpha$-strongly log-concave and $L$-log-smooth ($\alpha I \preceq -\nabla^2 \log \nu \preceq LI$) for some $0 < \alpha \le L$.
Let $\rho_k$, $\rho_k'$ be any two distributions evolving following the ULA algorithm~\eqref{Eq:ULA}.
Then for $0 < \epsilon \le \frac{2}{\alpha+L}$,
$W_2(\rho_k,\rho_k')^2 \le \big(1-2\epsilon\frac{\alpha L}{\alpha+L}\big)^k \, W_2(\rho_0,\rho_0')^2.$
\end{lemma}

The above implies that ULA has a unique stationary measure $\nu_\epsilon$.
The bias should be of order $1$
since ULA is a first-order discretization method, as suggested by Example~\ref{Ex:OU2}.
Using a simple synchronous coupling argument with the same smoothness assumption as in Lemma~\ref{Lem:ULASynch} only yields a bias of order $\frac{1}{2}$~\citep[Theorem~1]{D17a}.
But with an additional smoothness assumption, an expansion within the synchronous coupling argument yields a bias of order $1$~\citep[Theorem~4]{DK17}. 
It is interesting to ask whether it is possible to prove a bias of order $1$ with the minimal smoothness assumption as in Lemma~\ref{Lem:ULASynch}.

\begin{lemma}\label{Lem:ULABias}
Suppose $\nu$ is $\alpha$-strongly log-concave and $L$-log-smooth, and $-\nabla^2 \log \nu$ is $M$-Lipschitz.
For $0 < \epsilon < \frac{2}{\alpha+L}$, the bias of ULA is $W_2(\nu_\epsilon,\nu) \le \frac{\epsilon}{\alpha}(\frac{1}{2} Mn + \frac{11}{5} \sqrt{L^3 n})$.
\end{lemma}

The bias from Lemma~\ref{Lem:ULABias} and the exponential contraction from Lemma~\ref{Lem:ULASynch} imply an $O(\frac{1}{\delta} \log \frac{1}{\delta})$ iteration complexity bound for sampling from $\nu$ up to Wasserstein error $O(\delta)$, 
by choosing $\epsilon = O(\delta)$ and running ULA for $k = \Omega(\frac{1}{\epsilon} \log \frac{1}{\delta})$ iterations to get $W_2(\rho_k,\nu) \le W_2(\rho_k,\nu_\epsilon) + W_2(\nu_\epsilon,\nu) = O(\epsilon)$.
On the other hand, for the ideal unbiased discretization of the Langevin dynamics, exponential contraction would imply a logarithmic iteration complexity bound.

\subsubsection{ULA as the Forward-Flow discretization of Langevin dynamics}

ULA is biased because it is the Forward-Flow (FFl) discretization of the Langevin dynamics~\eqref{Eq:LD}.
Concretely, we can write ULA~\eqref{Eq:ULA} as a composition of two operations:
\begin{subequations}\label{Eq:ULA2}
\begin{align}
x_{k+\frac{1}{2}} &= x_k - \epsilon \nabla f(x_k) \label{Eq:ULA2a} \\
x_{k+1} &= x_{k+\frac{1}{2}} + \sqrt{2\epsilon} z_k.  \label{Eq:ULA2b}
\end{align}
\end{subequations}
The first step~\eqref{Eq:ULA2a} is a gradient descent step or the forward method for $f$;
the second step~\eqref{Eq:ULA2b} is the exact solution for the heat flow.
In the space of measures, the iterations in~\eqref{Eq:ULA2} correspond to
\begin{subequations}\label{Eq:ULA3}
\begin{align}
\rho_{k+\frac{1}{2}} &= (\I - \epsilon \nabla f)_\# \rho_k \label{Eq:ULA3a} \\
\rho_{k+1} &= \N(0,2\epsilon I) \ast \rho_{k+\frac{1}{2}} \label{Eq:ULA3b} 
\end{align}
\end{subequations}
where $\phantom{}_\#$ is the pushforward operator and $\ast$ is the convolution.
When $f$ is smooth ($\nabla^2 f \preceq \frac{1}{\epsilon} I$), 
the first step~\eqref{Eq:ULA3a} is the gradient descent for the expected function value $\E_\rho[f]$ (see~$\S\ref{Sec:GDF}$).
The second step~\eqref{Eq:ULA3b} is the exact gradient flow for negative entropy $-H(\rho)$ (see~$\S\ref{Sec:Heat}$).
These are the two components of the decomposition~\eqref{Eq:RelEnt2} of relative entropy. 
Therefore, ULA---which in the space of measures takes the form~\eqref{Eq:ULA3}---is the Forward-Flow method applied to the composite optimization problem of minimizing relative entropy~\eqref{Eq:RelEnt2}.
The source of the bias is that the flow method is not the adjoint of the forward method,
so the Foward-Flow method does not conserve the stationary point of the overall flow, which is the target measure $\nu$.
See~$\S\ref{App:Comp}$ for a review on composite optimization.

In a recent work, \cite{B18} proposes the proximal version of ULA, obtained by replacing the forward method (gradient descent) in~\eqref{Eq:ULA3a} with the backward method (proximal gradient) for $\E_\rho[f]$, which is also implemented by the proximal gradient for $f$ in space.
This is the Backward-Flow discretization of the Langevin dynamics,
which has similar convergence guarantees as ULA,
but is also still biased.

\subsection{The Forward-Backward method for Langevin dynamics}

A general algorithm for a composite optimization problem is the Forward-Backward (FB) method,\footnote{Or its adjoint, the Backward-Forward (BF) method.
However, the forward method for the heat flow is also not implementable, except for Gaussian data. 
We can also run the forward method (gradient descent) or backward method (proximal gradient) for OU with Gaussian data, and see they are consistent; 
see Examples~\ref{Ex:OUF} and~\ref{Ex:OUB}.}
which means running the forward method (gradient descent) for one component and running the backward method (proximal gradient) for the other.
The FB algorithm is consistent (unbiased) because the backward method is the adjoint of the forward method, so the minimizer is conserved; see $\S\ref{App:WhyFB}$.
Furthermore, FB converges exponentially fast under the gradient-domination condition and some smoothness assumptions~\cite[]{GRV17}, which is formally applicable in our case.
Therefore, we wish to run the FB method for composite optimization.

However, for our problem of optimizing relative entropy, FB means we need to replace~\eqref{Eq:ULA3b} with the backward method for negative entropy in the space of measures.
This is well-defined, but cannot be solved explicitly---unlike the exact flow~\eqref{Eq:ULA3b}---except for Gaussian initial data~\cite[]{CG03}.
In general this exception does not help, since Gaussianity is not preserved under the forward step~\eqref{Eq:ULA3a} when $\nabla f$ is nonlinear.
However, for the OU process, $\nabla f$ is linear, so a Gaussian initial data stays Gaussian under the FB algorithm, and in this case we see FB is indeed consistent; see Example~\ref{Ex:OUFB} in~$\S\ref{Sec:FBLang}$.

\subsection{The symmetrized Langevin algorithm}

Now that we know the unbiased FB algorithm for the Langevin dynamics is not implementable in general, we can try to reduce the bias of what we can implement.

Recall that the bias of an optimization algorithm is of the same size as its discretization order
(see $\S\ref{App:Disc}$ for a review of discretization methods).
Recall also that if an algorithm is symmetric, which means it is equal to its adjoint, then its order is even; in particular, it must be of order at least $2$.
In general we can symmetrize an algorithm by composing it with its adjoint. 
Therefore, given any algorithm with a first-order bias, we can upgrade it to an algorithm with a second-order bias by symmetrizing it.

Applying this idea to FFl as the base algorithm, we obtain the {\em symmetrized Forward-Flow} (SFFl) algorithm, which is the composition of FFl and its adjoint, the Flow-Backward (FlB) algorithm.

Applying SFFl to the Langevin dynamics yields the {\em symmetrized Langevin algorithm} (SLA):
\begin{align}\label{Eq:SLA}
x_{k+1} = (I + \epsilon \nabla f)^{-1}(x_k-\epsilon \nabla f(x_k) + \sqrt{4\epsilon} z_k)
\end{align}
where $z_k \sim \N(0,I)$ is independent of $x_k$.
Here $(I+\epsilon \nabla f)^{-1}$ is the proximal gradient operator of $f$, i.e., $y = (I+\epsilon \nabla f)^{-1}(x)$ if and only if $y + \epsilon \nabla f(y) = x$, or $y = \arg\min_{y' \in \R^n} \{ f(y') + \frac{1}{2\epsilon} \|y'-x\|^2\}$.
This is not analytically solvable for general $f$, and may require numerical computation in each iteration.
However, if we can do this, then we can run SLA which in principle has a smaller bias.
When the target measure is Gaussian, namely for the OU process, we can write the SLA iteration explicitly.
In this case we see that SLA is in fact consistent,\footnote{
This is surprising because SFFl is biased even for minimizing a sum of two quadratic functions in space. 
Here the consistency of SLA for OU relies on the property that variance adds linearly when we sum independent Gaussians.
}
and converges exponentially fast.

\begin{example}[SLA for OU.]\label{Ex:OUSLA}
Let $\nu = \N(\mu,\Sigma)$.
The SLA iteration is 
$$x_{k+1}-\mu = (I+\epsilon \Sigma^{-1})^{-1}(I-\epsilon \Sigma^{-1})(x_k-\mu) + \sqrt{4\epsilon} (I+\epsilon \Sigma^{-1})^{-1} z_k.$$
Unfolding and using the fact that the sum of independent Gaussians is Gaussian, we can write
$$x_k - \mu \stackrel{d}{=} A_\epsilon^k(x_0-\mu) + \sqrt{4\epsilon} B_\epsilon (I-A_\epsilon^2)^{-\frac{1}{2}} (I-A_\epsilon^{2k})^{\frac{1}{2}} z$$
where $A_\epsilon = (I+\epsilon \Sigma^{-1})^{-1}(I-\epsilon \Sigma^{-1})$, $B_\epsilon = (I+\epsilon \Sigma^{-1})^{-1}$, and $z \sim \N(0,I)$ is independent of $x_0$.
For all $\epsilon > 0$, $\lim_{k \to \infty} A_\epsilon^k = 0$.
Therefore, $x_k \stackrel{d}{\to} \mu + \sqrt{4\epsilon} B_\epsilon (I-A_\epsilon^2)^{-\frac{1}{2}} z \sim \N(\mu,\Sigma)$.
This shows SLA converges to the correct target measure $\nu = \N(\mu,\Sigma)$.
\end{example}

We also note that there are other discretizations of the Langevin dynamics that are unbiased for the Gaussian target measure, for example the Ozaki discretization which uses Hessian information; see for example~\citep{D17}.
We can also implement SLA for a mixture of two Gaussians, which requires solving a one-dimensional numerical problem in each iteration;
see~$\S\ref{App:MG}$.

\subsubsection{Convergence to the biased limit}

Similar to ULA, we have the following contraction result for SLA under strong log-concavity; see~Appendix~\ref{App:SLASynch}.

\begin{lemma}\label{Lem:SLASynch}
Suppose $\nu$ is $\alpha$-strongly log-concave and $L$-log-smooth ($\alpha I \preceq -\nabla^2 \log \nu \preceq LI$) for some $0 < \alpha \le L$.
Let $\rho_k$, $\rho_k'$ be any two distributions evolving following the SLA algorithm~\eqref{Eq:SLA}.
Then for $0 < \epsilon \le \frac{2}{\alpha+L}$,
$$W_2(\rho_k,\rho_k')^2 \le \Bigg(\frac{1-2\epsilon\frac{\alpha L}{\alpha+L}}{1+2\epsilon\frac{ \alpha L}{\alpha+L}}\Bigg)^k \, W_2(\rho_0,\rho_0')^2.$$
\end{lemma}

The above implies that SLA has a unique stationary measure $\tilde \nu_\epsilon$.
Despite Example~\ref{Ex:OUSLA}, in general SLA is biased, and the bias should be of order 2 since SLA is a second-order discretization method.
However, using the expansion within synchronous coupling as in ULA seems to still yield a bias of order 1 due to the stochastic terms. 

\begin{question}\label{Q:SLA}
Is it true that if $\nu$ is strongly log-concave, then $W_2(\nu, \tilde \nu_\epsilon) = O(\epsilon^2)$?
\end{question}

If Question~\ref{Q:SLA} is true, then combined with Lemma~\ref{Lem:SLASynch}, it implies an $O(\frac{1}{\sqrt{\delta}} \log \frac{1}{\delta})$ iteration complexity bound for sampling from $\nu$ up to Wasserstein error $O(\delta)$, by choosing $\epsilon = O(\sqrt{\delta})$ and running SLA for $k = O(\frac{1}{\epsilon} \log \frac{1}{\delta})$ iterations to get $W_2(\rho_k,\nu) \le W_2(\rho_k,\tilde\nu_\epsilon) + W_2(\tilde \nu_\epsilon,\nu) = O(\epsilon^2) = O(\delta)$.

We also note that a bias of order 2 may be the best we can hope for, since in general any higher order necessitates running the algorithm with negative step size~\cite[Theorem~3.18]{HLW06}.
For the Langevin dynamics, this means running the heat flow backward in time, which is not only non-implementable, but also not well-posed mathematically (except in special cases, e.g., see~\cite{M61}).
But as entropy is a very special functional, it is possible this apparent difficulty may be circumvented and implementable algorithms with higher-order bias may be found.

\section{Optimization in the space of measures}
\label{Sec:OptMeas}

Let us discuss further how to optimize in the space of measures.
We review optimization in a smooth Riemannian manifold in $\S\ref{App:OptM}$, including gradient flow and the forward and backward methods, and conditions for exponential convergence such as gradient domination and strong convexity.
We also briefly recall the Wasserstein metric; see~$\S\ref{App:Wass}$ for a review or~\cite{Vil08} for more detail.

Let $\P \equiv \P_2(\R^n)$ denote the space of probability measures on $\R^n$ with finite second moments,
endowed with the Wasserstein metric $W_2$ induced by the quadratic distance.
Every element $\rho \in \P$ is a probability measure represented by its density function $\rho \colon \R^n \to \R$ with respect to the Lebesgue measure $dx$.
The tangent space $\T_\rho \P$ consists of functions $R \colon \R^n \to \R$ of the form $R = -\nabla \cdot (\rho \nabla \phi)$ for some $\phi \colon \R^n \to \R$;
we write $R \equiv \nabla \phi$.
The norm of $R \equiv \nabla \phi$ is $\|R\|_\rho = (\E_\rho[\|\nabla \phi\|^2])^{\frac{1}{2}}$.

Suppose we wish to solve the optimization problem
$$\min_{\rho \in \P} F(\rho)$$
where $F \colon \P \to \R$ is a smooth functional.
There are two basic classes of interesting functionals:

\begin{enumerate}
  \item {\bf Expected value:} 
  Suppose $F(\rho) = \E_\rho[f]$ is the expected value of a smooth function $f \colon \R^n \to \R$.
  This is an example of a ``potential energy''~\cite[$\S5.2.2$]{Vil03}.
  It involves only a scalar product of a density $\rho$ and a function $f$,
  so can be implemented by samples in space.
  The gradient flow, gradient descent, and proximal gradient methods for $\E_\rho[f]$ are implemented by the corresponding gradient flow, gradient descent, and proximal gradient methods for $f$, with the same convergence guarantees.
  Therefore, we can view a deterministic problem from the space of measures without loss of information.
  See $\S\ref{App:ExpVal}$ for more detail.
  
  \item {\bf Negative entropy:} 
  Suppose $F(\rho) = -H(\rho)$ is the negative entropy.
  This is an example of an ``internal energy'' where we apply a function to the density $\rho$ before integrating,
  so it apparently cannot be implemented from samples in space.
  Interestingly, the gradient flow of negative entropy is the heat flow, which is implementable by the Brownian motion (Gaussian noise) in space.
  This also gives an optimization interpretation of the Fisher information as the squared gradient of entropy.  
  See $\S\ref{App:Ent}$ for more detail.
\end{enumerate}

There is a third class of functionals
which is ``interaction energy'', for example variance.
We do not use interaction energy for sampling in this paper, but see~$\S\ref{App:Var}$ for the gradient flow of variance.

\subsection{Minimizing relative entropy}
\label{Sec:RelEnt}

Our objective function for sampling is the {\em relative entropy}, which is a combination of the potential and internal energies:
\begin{align}\label{Eq:FRelEnt}
F(\rho) = H_\nu(\rho) = \int \rho \log \frac{\rho}{\nu} = \E_\rho[f] - H(\rho)
\end{align}
where $\nu = e^{-f}$ is the target measure.
Relative entropy is nonnegative, $H_\nu(\rho) \ge 0$,
and it is minimized at the target measure: $H_\nu(\rho) = 0$ if and only if $\rho = \nu$.

\subsubsection{Log-Sobolev inequality as gradient domination of relative entropy}
\label{Sec:LSI}

The squared gradient of relative entropy $F(\rho) = H_\nu(\rho)$ is the {\em relative Fisher information}:
$$J_\nu(\rho) = \E_\rho\left[\left\|\nabla \log \frac{\rho}{\nu} \right\|^2\right].$$
Note that $\rho = \nu$ is the only stationary point of $H_\nu$;
because if $\rho$ is a stationary point of $H_\nu$, then $\|\grad_\rho H_\nu\|^2_\rho = J_\nu(\rho) = 0$, and it is clear that $J_\nu(\rho) = 0$ if and only if $\rho = \nu$.

The gradient domination condition $\|\grad F\|^2 \ge 2\alpha (F - \min F)$, $\alpha > 0$, for relative entropy $F = H_\nu$ becomes the {\em logarithmic Sobolev inequality} (LSI)~\cite[]{G75,OV00}:
$$J_\nu(\rho) \ge 2\alpha H_\nu(\rho) ~~~~ \forall \, \rho \in \P.$$
The gradient flow identity $\frac{d}{dt} F(\rho) = -\|\grad_\rho F\|^2_\rho$ becomes $\frac{d}{dt} H_\nu(\rho) = -J_\nu(\rho)$, which is a generalization of the De Bruijn's identity.

The set of measures satisfying LSI includes all strongly log-concave measures,
and it is closed under bounded perturbation with a constant that decays exponentially with the size of the perturbation~\cite[]{HS87}. 
Thus, even a multimodal distribution such as a mixture of Gaussian satisfies LSI, and hence the Langevin dynamics converges exponentially fast. 

The Hessian of relative entropy $F(\rho) = -H_\nu(\rho)$ is, for a tangent function $R \equiv \nabla \phi \in \T_\rho \P$,
$$(\Hess_\rho H_\nu)(R,R) = \E_\rho \big[\|\nabla^2 \phi\|^2_{\HS} + \langle \nabla \phi, (\nabla^2 f) \nabla \phi \rangle \big].$$
Therefore, if $f$ is strongly convex ($\nu$ is strongly log-concave), then $F = H_\nu$ is also strongly convex, in which case any two co-evolving solutions are contracting exponentially fast.
However, note that the Hessian of $H_\nu$ is not bounded above.

\subsubsection{Langevin dynamics as gradient flow of relative entropy}

The gradient of $F(\rho) = H_\nu(\rho)$ is $\grad_\rho F = -\nabla \cdot (\rho \nabla \log \frac{\rho}{\nu}) \equiv \nabla \log \frac{\rho}{\nu}$.
Therefore, the gradient flow equation $\dot \rho = -\grad_\rho F$ of relative entropy is the Fokker-Planck equation~\eqref{Eq:FP}:
\begin{align}\label{Eq:FP2}
\part{\rho}{t} \,=\, \nabla \cdot \left(\rho \nabla \log \frac{\rho}{\nu}\right) \,=\, \nabla \cdot (\rho \nabla f) + \Delta \rho.\end{align}
This is implementable in space as the Langevin dynamics stochastic differential equation~\eqref{Eq:LD}:
$$dX = -\nabla f(X) \, dt + \sqrt{2} \, dW.$$
When the target measure is Gaussian, namely for the Ornstein-Uhlenbeck (OU) process, we have an exact solution, as we have seen in Example~\ref{Ex:OU}. 
However, in general we need to discretize.

If $\nu$ satisfies $\alpha$-LSI, then along the gradient flow~\eqref{Eq:FP2}, 
$$\frac{d}{dt} H_\nu(\rho) = -J_\nu(\rho) \le -2\alpha H_\nu(\rho),$$
which implies exponential convergence in relative entropy: 
$$H_\nu(\rho_t) \le e^{-2\alpha t} H_\nu(\rho_0).$$
This also implies exponential convergence of the distance $W_2(\rho_t,\nu)$, since LSI implies the 
{\em Talagrand inequality}:
$W_2(\rho,\nu)^2 \le \frac{2}{\alpha} H_\nu(\rho), \forall\,\rho \in \P$~\cite[]{T96,OV00}.
This is because in general gradient domination implies the sufficient growth property, see Appendix~\ref{Sec:CondExp}.

\subsubsection{Forward method for Langevin dynamics}
\label{Sec:FLang}

The forward method $\rho_{k+1} = \exp_{\rho_k}(-\epsilon \grad_{\rho_k} F)$ for relative entropy $F = H_\nu$ is
$$\rho_{k+1} = \exp_{\rho_k}\left(-\epsilon \nabla \log \frac{\rho_k}{\nu}\right)$$
where recall $\grad_\rho F \equiv \nabla \log \frac{\rho}{\nu}$.
If $\rho_k$ is $K$-log-semiconcave relative to $\nu$, which means $-\nabla^2 \log \frac{\rho_k}{\nu} \succeq KI$ for some $K \in \R$, then for $\epsilon \le \frac{1}{\max\{0,-K\}}$, the exponential map above is given by:
\begin{align}\label{Eq:FwLang}
\rho_{k+1} = \Big(I - \epsilon \nabla \log \frac{\rho_k}{\nu}\Big)_\# \rho_k.
\end{align}
If we know the analytic form of $\rho_k$, then we can implement one step of the algorithm
by $x_{k+1} = x_k - \epsilon \nabla f(x_k) - \epsilon \nabla \log \rho_k(x_k)$.
However, we cannot iterate this algorithm because at the next round we do not know what $\rho_{k+1}$ is, we only have $x_{k+1}$.

For Gaussian target measure, we can solve the forward method with Gaussian initial data.

\begin{example}[Forward method for OU with Gaussian data.]
\label{Ex:OUF}
Let $\nu = \N(\mu,\Sigma)$  as in Example~\ref{Ex:OU}.
Let $\rho_0 = \N(\mu,\Sigma_0)$ with $\Sigma_0^{-1} \succeq \Sigma^{-1}$. 
Along the forward method~\eqref{Eq:FwLang} for OU, $\rho_k = \N(\mu,\Sigma_k)$ stays Gaussian. 
Further,~\eqref{Eq:FwLang} becomes $x_{k+1}-\mu = (I+\epsilon(\Sigma_k^{-1}-\Sigma^{-1}))(x_k-\mu)$.
Therefore, $\Sigma_{k+1} = \Sigma_k(I+\epsilon(\Sigma_k^{-1}-\Sigma^{-1}))^2$.
The only fixed point is $\Sigma_k = \Sigma_{k+1} = \Sigma$. 
Thus, the forward method is consistent for OU with Gaussian data.
\end{example}

\subsubsection{Backward method for Langevin dynamics}
\label{Sec:BLang}

The backward method $\rho_{k+1} = \arg\min_{\rho \in \P} \{F(\rho) + \frac{1}{2\epsilon} W_2(\rho,\rho_k)^2\}$ for relative entropy  $F = H_\nu$ is
$$\exp_{\rho_{k+1}}\left(\epsilon \nabla \log \frac{\rho_{k+1}}{\nu}\right) = \rho_k.$$
If $\rho_{k+1}$ is $L$-log-smooth with respect to $\nu$, which means $-\nabla^2 \log \frac{\rho_{k+1}}{\nu} \preceq LI$ for some $L > 0$, then for $\epsilon \le \frac{1}{L}$, the backward method above is implemented (implicitly) by:
\begin{align}\label{Eq:BwLang}
\left(I + \epsilon \nabla \log \frac{\rho_{k+1}}{\nu}\right)_\# \rho_{k+1} = \rho_k.
\end{align}
In general this is not solvable analytically.

For Gaussian target measure, we can solve the backward method with Gaussian initial data.

\begin{example}[Backward method for OU with Gaussian data.]
\label{Ex:OUB}
Let $\nu = \N(\mu,\Sigma)$ as in Example~\ref{Ex:OU}.
Let $\rho_0 = \N(\mu,\Sigma_0)$ for simplicity.
Along the backward method~\eqref{Eq:BwLang} for OU, $\rho_k = \N(\mu,\Sigma_k)$ stays Gaussian.
Further,~\eqref{Eq:BwLang} becomes
$(I - \epsilon (\Sigma_{k+1}^{-1}-\Sigma^{-1})) (x_{k+1}-\mu) = x_k-\mu$.
Therefore, $\Sigma_{k+1}(I-\epsilon(\Sigma_{k+1}^{-1}-\Sigma^{-1}))^2 = \Sigma_k$.
The only fixed point is $\Sigma_k = \Sigma_{k+1} = \Sigma$.
Thus, the backward method is consistent for OU with Gaussian data.
\end{example}

\section{Langevin dynamics as composite optimization in the space of measures}
\label{Sec:Comp}

In~$\S\ref{Sec:RelEnt}$ we have seen how to optimize $H_\nu(\rho)$ by considering it as a single function.
We now study how to optimize $H_\nu(\rho)$ as a composite function when we write it as a sum of two functions:
$$H_\nu(\rho) = \E_\rho[f] - H(\rho).$$

\subsection{Forward-Backward for Langevin dynamics}
\label{Sec:FBLang}

In general, the algorithm of choice for composite optimization is the Forward-Backward (FB) algorithm, which means we run the forward method for one component and the backward method for the other.
(Equivalently, we can run the Backward-Forward algorithm, which is the adjoint version.)
The FB algorithm is consistent because the backward method is adjoint to the forward method, so the FB algorithm preserves the stationary point; see $\S\ref{App:Comp}$ for a review.
FB can be shown to converge exponentially fast under gradient domination condition and some smoothness assumptions~\cite[]{GRV17}; see also~$\S\ref{App:FBRn}$.
In principle, the FB algorithm for the Langevin dynamics is the answer we are seeking.

For optimizing relative entropy, the FB algorithm means running the forward method for $\E_\rho[f]$ (which is implemented by the gradient descent for $f$), followed by the backward method for $-H(\rho)$:
\vspace{-10pt}
\begin{subequations}\label{Eq:FBLang}
\begin{align}
\rho_{k+\frac{1}{2}} &= (I - \epsilon \nabla f)_\# \rho_k \label{Eq:FBLanga} \\
\rho_{k+1} &= \arg\min_{\rho \in \P} \Big\{ -H(\rho) + \frac{1}{2\epsilon} W_2(\rho,\rho_{k+\frac{1}{2}})^2\Big\}. \label{Eq:FBLangb}
\end{align}
\end{subequations}

We cannot implement the backward method for the heat flow in~\eqref{Eq:FBLangb}, so the FB algorithm above is not implementable in general.
For Gaussian target measure, we can solve FB for Gaussian initial data.
In this case the distributions stay Gaussian, and FB is indeed consistent;
see Example~\ref{Ex:OUFB} in $\S\ref{App:FBLang}$.

\subsection{Backward-Forward for Langevin dynamics}
\label{App:BFLang}

Similarly, we can run the Backward-Forward (BF) algorithm for the Langevin dynamics.

For optimizing relative entropy, BF means running the backward method for $\E_\rho[f]$ (which is implemented by the proximal gradient step for $f$), followed by the forward method for $-H(\rho)$:
\begin{subequations}\label{Eq:BFLang}
\begin{align}
\rho_{k+\frac{1}{2}} &= \left((I + \epsilon \nabla f)^{-1}\right)_\# \rho_k \label{Eq:BFLanga} \\
\rho_{k+1} &= \exp_{\rho_{k+\frac{1}{2}}}(-\epsilon \nabla \log \rho_{k+\frac{1}{2}}). \label{Eq:BFLangb}
\end{align}
\end{subequations}

The BF algorithm above is not implementable in general, since we cannot implement the forward method for the heat flow beyond one step.
For Gaussian target measure, we can solve BF for Gaussian initial data.
In this case the distributions stay Gaussian, and BF is indeed consistent;
see Example~\ref{Ex:OUBF} in $\S\ref{App:BFLang}$.

\section{Discussion and future work}
\label{Sec:Disc}

In this paper we have studied sampling as optimization in the space of measures. We started with the question of whether we can have a consistent discretization of the Langevin dynamics that converges exponentially fast under LSI. We have seen that the difficulties are twofold: First, relative entropy is a composite optimization problem in the space of measures, so we have to work with composite algorithms. Second, the heat flow is exactly solvable, but neither its forward nor backward methods are implementable. Therefore, unbiased algorithms such as the FB algorithm are not implementable for the Langevin dynamics. The basic discretization known as ULA is implementable but biased. We also proposed a symmetrized variant of ULA which should have a smaller bias, at the price of implementing the proximal gradient step in space.

We have focused on the Langevin dynamics, which is the gradient flow dynamics for minimizing relative entropy.
More generally, we can apply more sophisticated optimization techniques, such as acceleration, to sampling. 
There is a second-order variant of the Langevin dynamics known as the underdamped Langevin dynamics, which is the stochastic version of the second-order heavy ball dynamics for optimization, and has been shown to have better convergence properties than the Langevin dynamics~\citep{CCBJ17}.
However, it is interesting to consider whether we can also apply the acceleration principle directly in the space of measures, for example via the variational Lagrangian approach~\citep{WWJ16}.

\bibliography{wibisono18_arxiv.bbl}

\newpage
\appendix

\section{Details for~$\S\ref{Sec:Samp}$}

\subsection{Proof of Lemma~\ref{Lem:ULASynch} (Contraction of ULA)}
\label{App:ULASynch}

We use synchronous coupling to show iteratively that
$$W_2(\rho_{k+1},\rho_{k+1}')^2 \le \left(1-\frac{2\epsilon \alpha L}{\alpha+L}\right) \, W_2(\rho_k,\rho_k')^2$$
which will imply the desired claim.

Let $x_k \sim \rho_k$, $x_k' \sim \rho_k'$ be coupled with the optimal coupling, so $W_2(\rho_k,\rho_k') = \E[\|x_k-x_k'\|^2]$.
We evolve $x_k$, $x_k'$ via ULA~\eqref{Eq:ULA} with the same Gaussian noise $z_k$ (this is the synchronous coupling):
\begin{align*}
x_{k+1} &= x_k - \epsilon \nabla f(x_k) + \sqrt{2\epsilon} z_k \\
x_{k+1}' &= x_k' - \epsilon \nabla f(x_k') + \sqrt{2\epsilon} z_k.
\end{align*}
Subtracting and taking the squared norm, we obtain
\begin{align*}
\|x_{k+1}-x_{k+1}'\|^2 &= \|x_k - x_k' - \epsilon (\nabla f(x_k) - \nabla f(x_k'))\|^2 \\
&= \|x_k-x_k'\|^2 - 2\epsilon \langle \nabla f(x_k) - \nabla f(x_k'), x_k-x_k' \rangle + \epsilon^2 \|\nabla f(x_k) - \nabla f(x_k')\|^2.
\end{align*}
Since $f = -\log \nu$ is $\alpha$-strongly convex and $L$-smooth, we have by~\cite[Theorem~2.1.12]{Nesterov04}:
\begin{align}\label{Eq:Nest}
\langle \nabla f(x_k) - \nabla f(x_k'), x_k-x_k' \rangle \ge \frac{\alpha L}{\alpha+L} \|x_k-x_k'\|^2 + \frac{1}{\alpha+L} \|\nabla f(x_k)-\nabla f(x_k')\|^2.
\end{align}
Therefore, we have the bound
$$\|x_{k+1}-x_{k+1}'\|^2 \le \left(1-\frac{2\epsilon \alpha L}{\alpha+L}\right) \|x_k-x_k'\|^2 + \epsilon \left(\epsilon - \frac{2}{\alpha+L}\right)\|\nabla f(x_k)-\nabla f(x_k')\|^2.$$
If $0 < \epsilon \le \frac{2}{\alpha+L}$, then the second term above is nonpositive, so we may drop it:
$$\|x_{k+1}-x_{k+1}'\|^2 \le \left(1-\frac{2\epsilon \alpha L}{\alpha+L}\right) \|x_k-x_k'\|^2.$$
Now we take expectation and use the fact that $x_k$, $x_k'$ have the optimal coupling:
$$\E[\|x_{k+1}-x_{k+1}'\|^2] \le \left(1-\frac{2\epsilon \alpha L}{\alpha+L}\right) \E[\|x_k-x_k'\|^2]
=  \left(1-\frac{2\epsilon \alpha L}{\alpha+L}\right) W_2(\rho_k,\rho_k')^2.$$
Finally, by the definition of Wasserstein distance as the infimum over all coupling, we conclude
$$W_2(\rho_{k+1},\rho_{k+1}')^2 \le  \left(1-\frac{2\epsilon \alpha L}{\alpha+L}\right) W_2(\rho_k,\rho_k')^2$$
as desired.
\hfill$\square$

\subsection{Proof of Lemma~\ref{Lem:ULABias} (Bias of ULA)}

This follows from the noiseless case ($\delta = \sigma = 0$) of~\citep[Theorem~4]{DK17}.

\subsection{Proof of Lemma~\ref{Lem:SLASynch} (Contraction of SLA)}
\label{App:SLASynch}

We follow the same outline as~$\S\ref{App:ULASynch}$.
We use synchronous coupling to show iteratively that
$$W_2(\rho_{k+1},\rho_{k+1}')^2 \le \left(\frac{1-2\epsilon\frac{\alpha L}{\alpha+L}}{1+2\epsilon\frac{ \alpha L}{\alpha+L}}\right) \, W_2(\rho_k,\rho_k')^2$$
which will imply the desired claim.

Let $x_k \sim \rho_k$, $x_k' \sim \rho_k'$ be coupled with the optimal coupling, so $W_2(\rho_k,\rho_k') = \E[\|x_k-x_k'\|^2]$.
We evolve $x_k$, $x_k'$ via SLA~\eqref{Eq:SLA} with the same Gaussian noise $z_k$ (this is the synchronous coupling):
\begin{align*}
x_{k+1} + \epsilon \nabla f(x_{k+1}) &= x_k - \epsilon \nabla f(x_k) + \sqrt{2\epsilon} z_k \\
x_{k+1}' + \epsilon \nabla f(x_{k+1}') &= x_k' - \epsilon \nabla f(x_k') + \sqrt{2\epsilon} z_k.
\end{align*}
Subtracting and taking the squared norm, we obtain
\begin{align*}
\|x_{k+1}-x_{k+1} + \epsilon (\nabla f(x_{k+1}) - \nabla f(x_{k+1}')) \|^2 &= \|x_k - x_k' - \epsilon (\nabla f(x_k) - \nabla f(x_k'))\|^2.
\end{align*}
We expand both sides and use the inequality~\eqref{Eq:Nest}.
As in~$\S\ref{App:ULASynch}$, for $0 < \epsilon \le \frac{2}{\alpha+L}$, the right hand side is upper bounded by $(1-\frac{2\epsilon \alpha L}{\alpha+L}) \|x_k-x_k'\|^2$.
Similarly, the left hand side is lower bounded by $(1+\frac{2\epsilon \alpha L}{\alpha+L}) \|x_{k+1}-x_{k+1}'\|^2$.
Combining and taking expectation, we obtain
$$\E[\|x_{k+1}-x_{k+1}'\|^2] \le \left(\frac{1-2\epsilon\frac{\alpha L}{\alpha+L}}{1+2\epsilon\frac{ \alpha L}{\alpha+L}}\right) \E[\|x_k-x_k'\|^2] = \left(\frac{1-2\epsilon\frac{\alpha L}{\alpha+L}}{1+2\epsilon\frac{ \alpha L}{\alpha+L}}\right) W_2(\rho_k,\rho_k')^2.$$
Finally, by the definition of Wasserstein distance as the infimum over all coupling, we conclude
$$W_2(\rho_{k+1},\rho_{k+1}')^2 \le \left(\frac{1-2\epsilon\frac{\alpha L}{\alpha+L}}{1+2\epsilon\frac{ \alpha L}{\alpha+L}}\right) \, W_2(\rho_k,\rho_k')^2$$
as desired.
\hfill $\square$

\subsection{SLA and ULA for mixture of Gaussians}
\label{App:MG}

Let the target measure be an equal mixture of two Gaussians:
$$\nu = \frac{1}{2} \N(-a,I) + \frac{1}{2} \N(a,I)$$
for some $a \in \R^n$.
Then we can write $f = -\log \nu$ as
$$f(x) = \frac{1}{2}\|x\|^2-\log \cosh(\langle x,a \rangle) + \frac{1}{2} \|a\|^2 + \frac{n}{2} \log (2\pi).$$
The gradient of $f$ is $\nabla f(x) = x-\tanh(\langle x,a \rangle) a$.

\paragraph{ULA.}
The ULA iteration for mixture of Gaussians is:
$$x_{k+1} = (1-\epsilon) x_k + \epsilon \tanh(\langle x_k,a \rangle) a + \sqrt{2\epsilon} z_k$$
where $z_k \sim \N(0,I)$ is independent of $x_k$.
We can run this directly.

\paragraph{SLA.}
The ULA iteration for mixture of Gaussians is:
\begin{align}\label{Eq:SLAMixtGau}
(1+\epsilon) x_{k+1} - \epsilon \tanh(\langle x_{k+1},a \rangle ) a = (1-\epsilon) x_k + \epsilon \tanh(\langle x_k,a \rangle) a + \sqrt{4\epsilon} z_k
\end{align}
where $z_k \sim \N(0,I)$ is independent of $x_k$.

Let $v_k = \langle x_k, a \rangle \in \R$.
Taking the inner product of both sides of~\eqref{Eq:SLAMixtGau} with $a$ gives us
$$
(1+\epsilon) v_{k+1} - \epsilon \|a\|^2 \tanh(v_{k+1}) = (1-\epsilon) v_k + \epsilon \|a\|^2 \tanh(v_k) + \sqrt{4\epsilon} \langle z_k, a \rangle.
$$
Given $v_k$ and $z_k$, we can invert the equation above to solve for $v_{k+1}$, which is well-defined for small $\epsilon$.
Once we have $v_{k+1} = \langle x_{k+1}, a \rangle$, we can substitute it to~\eqref{Eq:SLAMixtGau} to solve for $x_{k+1}$:
$$x_{k+1} = \frac{1-\epsilon}{1+\epsilon} x_k + \frac{\epsilon}{1+\epsilon} \left( \tanh(v_k) + \tanh(v_{k+1})\right) a + \frac{\sqrt{4\epsilon}}{1+\epsilon} z_k.$$

\section{A review of discretization methods for a flow in space}
\label{App:Disc}

We provide brief review of discretization methods, and refer to~\cite[]{HLW06} for more detail.

\subsection{Integrator, order, and adjoint}

Let $\varphi = (\varphi_t)_{t \in \R}$ be the flow of the differential equation $\dot x = v(x)$ for a smooth vector field $v$ on $\R^n$.
An {\em integrator} for $\varphi$ is a family $A = (A_\epsilon)_{\epsilon \in \R}$ of algorithms $A_\epsilon \colon \R^n \to \R^n$, indexed by a {\em step size} $\epsilon \in \R$ (or $\epsilon$ in a neighborhood of $0$), such that
$A_0 = I$ is the identity map, $(\epsilon,x) \mapsto A_\epsilon(x)$ is smooth, and 
$$\lim_{\epsilon \to 0} \frac{A_\epsilon(x)-x}{\epsilon} = v(x) ~~~~ \forall \, x \in \R^n.$$
We say that the integrator $A$ has {\em order $p$}, for some $p \in \bN \cup \{+\infty\}$, if
$$\|\varphi_\epsilon(x) - A_\epsilon(x)\| \le O(\epsilon^{p+1})~~~~\text{ as } \epsilon \to 0, ~ \forall\, x \in \R^n$$
where the bound on the right hand side above may depend on $x$.

The {\em adjoint} of an integrator $A = (A_\epsilon)_{\epsilon \in \R}$ is another integrator $A^\ast = (A_\epsilon^\ast)_{\epsilon \in \R}$ defined by
$$A_\epsilon^\ast = (A_{-\epsilon})^{-1}.$$
(If $A_\epsilon$ is defined for $\epsilon$ in a symmetric interval around $0$, then $A_\epsilon^\ast$ is defined for $\epsilon$ in the same interval.)
This definition of adjoint satisfies $(A^\ast)^\ast = A$ and $(A \circ B)^\ast = B^\ast \circ A^\ast$ for any integrators $A$ and $B$, where $A \circ B = (A_{\epsilon} \circ B_{\epsilon})_{\epsilon \in \R}$ is the composition of two integrators. 

We say that an integrator $A$ is {\em symmetric} if it is self-adjoint: $A^\ast = A.$
We can symmetrize any integrator $A$ by composing it with its adjoint in either order; this gives two symmetric integrators $A \circ A^\ast$ and $A^\ast \circ A$, which in general are not equal.
We recall the following results.

\begin{lemma}[{\cite[Theorem~3.2]{HLW06}}]
\label{Lem:Int}
\begin{enumerate}
\item An integrator $A$ and its adjoint $A^\ast$ have the same order.
\item If $A$ is symmetric, then its order is even.
\end{enumerate}
\end{lemma}

For example, the exact flow $\varphi = (\varphi_t)_{t \in \R}$ is a symmetric integrator of order $+ \infty$.
This is the ideal integrator, but typically not computable in practice.
There are two first-order integrators that come from basic discretizations: the forward (explicit) method, and the backward (implicit) method.
We can also symmetrize them to obtain second-order integrators.
We describe them further below.

\subsection{The forward method}
\label{Sec:Fwd}
 
We wish to approximate the differential equation $\dot x = v(x)$.
The forward method 
$\F = (\F_\epsilon)_{\epsilon \in \R}$ 
uses the approximation $\frac{1}{\epsilon}(\F_\epsilon(x)-x) = v(x)$, or equivalently,
\begin{align}\label{Eq:Fw1}
\F_\epsilon(x) = x + \epsilon v(x).
\end{align}
We also write
\begin{align}\label{Eq:Fw2}
\F_\epsilon = I + \epsilon v
\end{align}
where $I$ is the identity map.
This is known as the forward or explicit Euler method, because given where we are now, we can determine where to go next with little computation.
If $v$ is smooth and Lipschitz, 
then the forward method 
$F_\epsilon$ is a diffeomorphism for $\epsilon$ in a small enough neighborhood around $0$, therefore $\F$ is an integrator.
Furthermore, $\F$ has order $1$ because $\F_\epsilon$ is performing a first-order approximation.

\subsection{The backward method}

We wish to approximate the differential equation $\dot x = v(x)$.
The backward method $\B = (\B_\epsilon)_{\epsilon \in \R}$ 
uses the approximation $\frac{1}{\epsilon}(\B_\epsilon(x)-x) = v(\B_\epsilon(x))$, or equivalently,
\begin{align}\label{Eq:Bw1}
\B_\epsilon(x) - \epsilon v(\B_\epsilon(x)) = x.
\end{align}
Therefore,
\begin{align}\label{Eq:Bw2}
\B_\epsilon = (I - \epsilon v)^{-1}.
\end{align}
This is known as the backward or implicit Euler method, because to determine where to go next we need to solve an implicit equation~\eqref{Eq:Bw1}, or equivalently compute the inverse of an operator~\eqref{Eq:Bw2}.
As in the forward method, if $v$ is smooth and Lipschitz, then for small $|\epsilon|$, the map $I-\epsilon v$ is a diffeomorphism.
So $\B_\epsilon = (I-\epsilon v)^{-1}$ is also a diffeomorphism, and therefore $\B$ is an integrator.

Furthermore, observe that we can write
$$\B_\epsilon = (\F_{-\epsilon})^{-1} = \F_\epsilon^\ast.$$
Therefore, the backward method is the adjoint of the forward method, and thus by Lemma~\ref{Lem:Int} they have the same order $1$.

\subsection{Symmetrized methods}

We can consider the symmetrized versions of the basic methods.
There are two versions: the symmetrized forward method---which applies the forward method followed by the backward method---is known as the trapezoid rule, while the symmetrized backward method---which applies the backward method followed by the forward method---is known as the implicit midpoint rule.
Both have order 2, being symmetric methods.
We can also consider compositions of basic methods with varying step sizes, properly chosen to increase the order of the resulting algorithm; see~\cite[$\S$II.II.4]{HLW06}.

\subsection{On order of error and order of bias}
\label{App:DiscOrder}

We show that the order of the discretization error of an algorithm is the same as the order of the bias, at least under strong convexity.
Thus, if we can find a discretization algorithm of high order, then we are guaranteed the bias will be small.

Suppose the vector field $v$ satisfies the following monotonicity property for some $\alpha > 0$:
\begin{align}\label{Eq:vCont}
\langle v(x)-v(y) , x-y \rangle \le -\alpha \|x-y\|^2~~~~\forall \, x,y \in \R^n.
\end{align}
This is satisfied, for example, for gradient flow of a strongly convex function, i.e., $v(x) = -\nabla f(x)$ where $f \colon \R^n \to \R$ is $\alpha$-strongly convex.
The condition above implies any two solutions of the flow of $v$ contract exponentially fast, and in particular there is a unique fixed point $x^\ast$.
We now show that any integrator for the flow of $v$ will also converge exponentially fast to a biased limit of the same order as the discretization error.

\begin{lemma}
Let $A = (A_\epsilon)_{\epsilon \in \R}$ be an integrator of order $p \in \bN$ for the flow of $v$ satisfying~\eqref{Eq:vCont}.
Let $x_0 \in \R^n$ and define iteratively $x_{k+1} = A_\epsilon(x_k)$.
Then
$$\|x_k-x^\ast\| \le O(\epsilon^p) + e^{-\alpha \epsilon k} \|x_0-x^\ast\|.$$
In particular, the bias of $A_\epsilon$ is $O(\epsilon^p)$.
\end{lemma}
\begin{proof}
Let $\tilde x(t) = \varphi_t(x_0)$ be the flow of $v$ starting from $\tilde x(0)$.
We compare $x_k$ with $\tilde x(\epsilon k)$.
By triangle inequality,
$$\|x_k-x^\ast\| \le \|x_k-\tilde x(\epsilon k)\| + \|\tilde x(\epsilon k)-x^\ast\|.$$
The first term is upper bounded by $e^{-\alpha \epsilon k} \|\tilde x(0) - x^\ast\|$ by property~\eqref{Eq:vCont}.
We need to show the second term $\delta_k := \|x_k - \tilde x(\epsilon k)\|$ is $O(\epsilon^p)$.

Consider the flow $\varphi_\epsilon(x_{k-1})$ of $v$ starting at $x_{k-1}$.
We also write $x_k = A_\epsilon(x_{k-1})$ and $\tilde x(\epsilon k) = \varphi_\epsilon(\tilde x(\epsilon (k-1)))$.
By triangle inequality,
$$\delta_k = \|x_k - \tilde x(\epsilon k)\| \le \|A_\epsilon(x_{k-1}) - \varphi_\epsilon(x_{k-1})\| + \|\varphi_\epsilon(x_{k-1}) - \varphi_\epsilon(\tilde x(\epsilon (k-1)))\|.$$
The first term above is $O(\epsilon^{p+1})$ since $A_\epsilon$ is an order-$p$ integrator.
The second term above is upper bounded by $e^{-\alpha \epsilon}\|x_{k-1} - \tilde x(\epsilon(k-1))\| = e^{-\alpha \epsilon} \delta_{k-1}$ by the contraction property of $\varphi_t$, by~\eqref{Eq:vCont}.
Then
$$\delta_k \le O(\epsilon^{p+1}) + e^{-\alpha \epsilon} \delta_{k-1}.$$
Unfolding the recursion with $\delta_0 = 0$ (since $\tilde x(0) = x_0$), we get
$$\delta_k \le \frac{O(\epsilon^{p+1})(1-e^{-\alpha \epsilon k})}{1-e^{-\alpha \epsilon}} \,=\, O(\epsilon^p)$$
as desired.
\end{proof}

\section{A review of optimization in Riemannian manifold}
\label{App:OptM}

We review optimization in a smooth Riemannian manifold.
We focus on gradient flow and simple discretization methods including the forward and backward methods.
We discuss sufficient conditions ensuring exponential convergence rate including gradient domination and strong convexity.

Let $\M$ be a complete smooth Riemannian manifold of dimension $n \ge 1$.
Let
$$f \colon \M \to \R$$
be a smooth objective function.
We assume $f$ is bounded below and achieves its minimum in $\M$ (not necessarily at a unique point), so the set of minimizers
$$x^\ast(f) = \arg\min_{x \in \M} f(x)$$
is not empty.
Let $\min f \equiv \min_{x \in \M} f(x)$ be the minimum value of $f$.

We want to solve the optimization problem
$$\min_{x \in \M} f(x)$$
which is equivalent to the problem of reaching the set of minimizers $x^\ast(f)$.
The basic dynamics that achieves this task is gradient flow.
We first recall some conditions that ensure exponential convergence rate.

\subsection{Conditions ensuring exponential convergence}
\label{Sec:CondExp}

We say that $f$ is {\em $\alpha$-strongly convex} for some $\alpha > 0$ if 
$$\Hess_x f \succeq \alpha I~~~\forall\, x \in \M.$$
Here recall the Hessian $\Hess_x f$ is the bilinear form on the tangent space $\T_x\M$ that measures the second rate of change of $f$ along geodesics, and the condition above means $(\Hess_x f)(v) \equiv (\Hess_x f)(v,v) \ge \alpha \|v\|_x^2$ for all $v \in \T_x\M$.

We say that $f$ is {\em $\alpha$-gradient dominated} for some $\alpha > 0$ if
$$\|\grad_x f\|^2_x \ge 2\alpha (f(x)-\min f)~~~\forall\, x \in \M.$$
For $\M = \R^n$, this is known as the Polyak-\L{}ojaciewicz condition or the Kurdyka-\L{}ojaciewicz condition~\cite[]{P63,L63}.
In the setting of $\M = \P_2(\R^n)$ with the relative entropy functional, this is known as the logarithmic Sobolev inequality~\cite[]{OV00}.
Observe that if $f$ is gradient dominated, then any stationary point of $f$ (where the gradient vanishes) must be a global minimum.

We say that $f$ has {\em $\alpha$-sufficient growth} for some $\alpha > 0$ if
$$f(x)-\min f \ge \frac{\alpha}{2} d(x,x^\ast(f))^2~~~\forall\, x \in \M.$$ 
Here $d(x,x^\ast(f)) = \inf_{x^\ast \in x^\ast(f)} d(x,x^\ast)$ is the distance from $x$ to the minimizer set $x^\ast(f)$.
In the setting of $\M = \P_2(\R^n)$ with the relative entropy functional, this is known as the Talagrand inequality~\cite[]{OV00}.

We recall the following implications from~\cite[Propositions~1' \& 2']{OV00}, 
which follow by interpolating the inequalities along the gradient flow.

\begin{lemma}\label{Lem:SC}
Let $\alpha > 0$.
\begin{enumerate}
  \item If $f$ is $\alpha$-strongly convex, then $f$ is $\alpha$-gradient dominated.
  \item If $f$ is $\alpha$-gradient dominated, then $f$ has $\alpha$-sufficient growth.
\end{enumerate}
\end{lemma}

\subsection{Gradient flow}

The basic dynamics for minimizing a function $f$ is the {\em gradient flow}:
\begin{align}\label{Eq:GFGen}
\dot x = -\grad_x f.
\end{align}
Here $\grad_x f \in \T_x\M$ is the metric gradient of $f$ at $x$, which in local coordinate is given by multiplying the vector of partial derivatives of $f$ by the inverse metric at $x$.
Gradient flow is a descent flow:
$$\frac{d}{dt} f(x) = \langle \grad_x f, \dot x \rangle_x = -\|\grad_x f\|^2_x \le 0.$$
We are interested in quantifying how fast the convergence occurs.

\subsubsection{Exponential contraction of solutions under strong convexity}

Strong convexity is the weakest condition needed for exponential contraction between solutions.
Indeed, suppose $f$ is $\alpha$-strongly convex for some $\alpha > 0$, and
let $x(t)$, $y(t)$ be two solutions of the gradient flow~\eqref{Eq:GFGen}.
Assume $x(t)$ and $y(t)$ are close enough so that the geodesic connecting them is minimizing.
Then the gradient of the squared distance $x \mapsto d(x,y(t))^2$ is $\grad_{x(t)} d(\cdot, y(t))^2 = -2\log_{x(t)}(y(t))$ (see for example~\cite[]{F06}).
Here $\log_x y \in \T_x M$ is the logarithm map, which is the inverse of the exponential map, i.e., $v = \log_x y$ if and only if $y = \exp_x(v)$.
Similarly, the gradient of $y \mapsto d(y,x(t))^2$ is $\grad_{y(t)} d(\cdot, x(t))^2 = -2\log_{y(t)}(x(t))$.
Then we have (hiding dependence on $t$ for simplicity):
\begin{align}
\frac{d}{dt} d(x(t),y(t))^2 
&= -2 \big\langle \log_{x}(y), \dot x \big\rangle_{x} - 2 \big\langle \log_{y}(x), \dot y \big\rangle_{y} \notag \\
&=  2 \big\langle \log_{x}(y), \grad_{x} f \big\rangle_{x} + 2 \big\langle \log_{y}(x), \grad_{y} f \big\rangle_{y}.  \label{Eq:GFExp}
\end{align}
On the other hand, by the strong convexity of $f$, we have
\begin{align*}
f(y) &\ge f(x) + \langle \grad_{x} f, \log_x y \rangle_x + \frac{\alpha}{2} d(x,y)^2 \\
f(x) &\ge f(y) + \langle \grad_{y} f, \log_y x \rangle_x + \frac{\alpha}{2} d(x,y)^2.
\end{align*}
Summing yields $\langle \grad_{x} f, \log_x y \rangle_x + \langle \grad_{y} f, \log_y x \rangle_x \le -\alpha d(x,y)^2$.
Substituting to~\eqref{Eq:GFExp} yields
$$\frac{d}{dt} d(x(t),y(t))^2 \le -2\alpha d(x(t),y(t))^2.$$
Therefore, $d(x(t),y(t))^2 \le e^{-2\alpha t} d(x(0),y(0))^2$, as desired.

\subsubsection{Exponential convergence of function value under gradient domination}

Gradient domination is the weakest condition needed for exponential convergence of function value.
Indeed, suppose $f$ is $\alpha$-gradient dominated for some $\alpha > 0$.
Then along the gradient flow,
$$\frac{d}{dt} (f(x) - \min f) = -\|\grad_x f\|^2_x \le -2\alpha (f(x)-\min f).$$
Therefore, $f(x(t)) - \min f \le e^{-2\alpha t} (f(x(0)) - \min f)$.
This also implies exponential convergence to the minimizer, since gradient domination implies sufficient growth.

\subsection{Gradient descent}

Gradient descent algorithm for minimizing $f$ with step size $\epsilon > 0$ is the iteration
$$x_{k+1} = \exp_{x_k}(-\epsilon \grad_{x_k} f).$$
Here we assume $\epsilon$ is small enough so the geodesic from $x_k$ to $x_{k+1}$ is minimizing.

As in gradient flow, in general we can get exponential contraction between solutions of gradient descent under strong convexity, and get exponential convergence in function value under gradient domination.
However, we now also need a smoothness assumption on $f$.
Specifically, suppose $f$ is $\alpha$-gradient dominated for some $\alpha > 0$.
Assume further $f$ is $L$-smooth for some $L > 0$, which means $\Hess f \preceq LI$.
This implies, for $x,y$ sufficiently close,
$$f(y) \le f(x) + \langle \grad_x f, \log_x y \rangle_x + \frac{L}{2} d(x,y)^2.$$
Plugging in $x = x_k$, $y = x_{k+1}$ along gradient descent with $d(x,y)^2 = \epsilon^2 \|\grad_{x_k} f\|_{x_k}^2$, we obtain
\begin{align}\label{Eq:GDEq}
f(x_{k+1}) \le f(x_k) - \epsilon \left(1-\frac{\epsilon L}{2}\right) \|\grad_{x_k} f\|^2_{x_k}.
\end{align}
If $0 < \epsilon \le \frac{2}{L}$, then we can chain the last term above with the gradient domination inequality to conclude that
$$f(x_{k+1}) - \min f \le \left(1-2\alpha \epsilon \left(1-\frac{\epsilon L}{2}\right) \right) (f(x_k)-\min f).$$
Unrolling the recursion gives the exponential convergence in function value:
$$f(x_k)-\min f \le \left(1-2\alpha \epsilon \left(1-\frac{\epsilon L}{2}\right) \right)^k (f(x_0)-\min f).$$
For example, if $\epsilon = \frac{1}{L}$, then the rate is $(1-\frac{\alpha}{L})^k$.
This also implies exponential convergence in the distance to minimizer by the sufficient growth property.
See also, for example,~\cite[]{FO98,ZS16,KNS16}.

\subsection{Proximal gradient}

The proximal gradient method for minimizing $f$ with step size $\epsilon > 0$ is
\begin{align}\label{Eq:PGGen}
x_{k+1} = \arg\min_{x \in M} \left\{ f(x) + \frac{1}{2\epsilon} d(x,x_k)^2 \right\}.
\end{align}
Assume $\epsilon > 0$ is small enough so the minimizer above is unique and within the injectivity radius of $x_k$.
Then the minimizer $x_{k+1}$ is characterized by $\grad_{x_{k+1}} (f + \frac{1}{2\epsilon} d(\cdot, x_k)^2) = 0$, or equivalently,
$$\log_{x_{k+1}} x_k = \epsilon \grad_{x_{k+1}} f.$$
This is equivalent to the implicit update $\exp_{x_{k+1}}(\epsilon \grad_{x_{k+1}} f) = x_k$, which generalizes the usual update $x_{k+1} + \epsilon \nabla f(x_{k+1}) = x_k$ in the Euclidean case.

As before, in general we can get exponential contraction between solutions under strong convexity, and get exponential convergence in function value under gradient domination.
Unlike in gradient descent, here we do not need a smoothness assumption on $f$, but we need to solve the implicit update above.

Suppose $f$ is $\alpha$-gradient dominated for some $\alpha > 0$.
Since $x_{k+1}$ is the minimizer of~\eqref{Eq:PGGen},
$$f(x_{k+1}) + \frac{1}{2\epsilon} d(x_{k+1},x_k)^2 \le f(x_k).$$
Equivalently, since $d(x_{k+1},x_k)^2 = \epsilon^2 \|\grad_{x_{k+1}} f\|^2_{x_{k+1}}$,
$$f(x_{k+1}) - f(x_k) \le -\frac{1}{2\epsilon} d(x_{k+1},x_k)^2 = -\frac{\epsilon}{2} \|\grad_{x_{k+1}} f\|^2_{x_{k+1}}.$$
Now using the gradient domination inequality and collecting the terms give us
\begin{align}\label{Eq:PGEq}
f(x_{k+1}) - \min f \le \frac{1}{1 + \alpha \epsilon} (f(x_k)-\min f).
\end{align}
Unfolding the recursion, we conclude that
$$f(x_k)-\min f \le \frac{1}{(1+\alpha \epsilon)^k} (f(x_0)-\min f).$$
See also, for example,~\cite[]{FO02,BEtAl16}

\subsection{Symmetrized forward method}

The symmetrized forward method for minimizing $f$ is the composition of the gradient descent and the proximal gradient methods:
\begin{subequations}
\begin{align}
x_{k+\frac{1}{2}} &= \exp_{x_k}(-\epsilon \grad_{x_k} f) \label{Eq:SGb} \\
x_{k+1} &= \arg\min_{x \in M} \left\{ f(x) + \frac{1}{2\epsilon} d(x,x_{k+\frac{1}{2}})^2 \right\}. \label{Eq:SGa}
\end{align}
\end{subequations}
Note that this is the Forward-Backward algorithm for composite optimization, applied to the self-decomposition $2f = f+f$.

Suppose $f$ is $\alpha$-gradient dominated and $L$-smooth for some $0 < \alpha \le L$, and let $0 < \epsilon \le \frac{2}{L}$.
Then as in~\eqref{Eq:GDEq}, the first update~\eqref{Eq:SGb} above implies
$$f(x_{k+\frac{1}{2}}) - f(x_{k+1}) \le  -\epsilon \left(1-\frac{\epsilon L}{2}\right) \|\grad_{x_k} f\|^2_{x_k}.$$
Similarly, as in~\eqref{Eq:PGEq}, the second update~\eqref{Eq:SGa} above implies
$$f(x_{k+1}) - f(x_{k+\frac{1}{2}}) \le -\frac{\epsilon}{2} \|\grad_{x_{k+1}} f\|^2_{x_{k+1}}.$$
Summing the two inequalities above and using the gradient domination inequality give us
$$f(x_{k+1})-f(x_k) \le -2\alpha \epsilon \left(1-\frac{\epsilon L}{2}\right)(f(x_k)-\min f) - \alpha \epsilon (f(x_{k+1})-\min f).$$
Collecting terms and unrolling the recursion, we conclude the exponential convergence rate
$$f(x_k)-\min f \le \left(\frac{1-2\alpha \epsilon \left(1-\frac{\epsilon L}{2}\right)}{1+\alpha \epsilon} \right)^k (f(x_0) - \min f).$$
For example, if $\epsilon = \frac{1}{L}$, then the rate is $\left(\frac{1-\frac{\alpha}{L}}{1+\frac{\alpha}{L}}\right)^k$, which is slightly better than for gradient descent.

\section{A brief review of the Wasserstein metric in the space of measures}
\label{App:Wass}

Let $\P \equiv \P_2(\R^n)$ be the space of probability measures $\rho$ on $\R^n$ with finite second moments.
We provide a brief review of the Wasserstein metric, and refer the reader for more detail to~\cite{Vil03,Vil08,CG03,OV00}.

The Wasserstein metric $W_2$ on $\P$ is defined as
\begin{align}\label{Eq:Wass}
W_2(\rho,\nu)^2 = \inf_{X \sim \rho, Y \sim \nu} \E[\|X-Y\|^2]
\end{align}
where the infimum is over all coupling of random variables $(X,Y)$ with $X \sim \rho$ and $Y \sim \nu$.
This formally endows $\P$ with an infinite-dimensional smooth Riemannian metric.

A tangent vector/function $R \in \T_\rho P$ is of the form
$$R = -\nabla \cdot (\rho \nabla \phi)$$
for some function $\phi \colon \R^n \to \R$.
We also write $R \equiv \nabla \phi$.
The squared norm of $R \equiv \nabla \phi$ is
$$\|R\|^2_\rho = \|\nabla \phi\|^2_\rho = \int \rho \|\nabla \phi\|^2 = \E_\rho[\|\nabla \phi\|^2].$$

The gradient of a functional $F \colon \P \to \R$ is
$$\grad_\rho F = -\nabla \cdot \left( \rho \nabla \frac{\delta F}{\delta \rho} \right)$$
where $\frac{\delta F}{\delta \rho}(x)$ is the naive ($L^2$) derivative of $F$ with respect to $\rho(x)$.
There are also Hessian formulae for specific functional classes~\cite[$\S15$]{Vil08}.

When $\rho$ and $\nu$ are smooth and absolutely continuous, the optimal coupling in the Wasserstein distance~\eqref{Eq:Wass} is unique and induced by an optimal transport map which is the gradient of a convex function: $(X,Y) = (X, \nabla \phi(X))$, where $\phi \colon \R^n \to \R$ is convex and satisfies $(\nabla \phi)_\# \rho = \nu$.
Therefore, if there is a convex function whose gradient pushes forward one measure to another, then that gradient must be the optimal transport map.
Furthermore, the geodesic in the Wasserstein metric connecting one measure to another is obtained by linearly interpolating the optimal transport map in space and taking the pushforward. 
We summarize this fact in the following lemma, which we use in our discussion in this paper.

\begin{lemma}
Let $\phi \colon \R^n \to \R$ satisfy $\nabla^2 \phi \succeq KI$ for some $K \in \R$.
For $0 \le \epsilon < \frac{1}{\max\{0,-K\}}$,
the exponential map $\exp_\rho(\epsilon \nabla \phi)$ in the Wasserstein metric is given by the pushforward of the linear map $(I+\epsilon \nabla \phi)_\#$.
\end{lemma}
\begin{proof}
The map 
$I + \epsilon \nabla \phi = \nabla \left(\frac{\|\cdot\|^2}{2} + \epsilon \phi\right)$ is the gradient of a convex function that pushes forward $\rho$ to $\exp_\rho(\epsilon \nabla \phi)$, so it must be the optimal transport map.
\end{proof}

\section{Details for $\S\ref{Sec:OptMeas}$: Optimization in the space of measures} 

\subsection{Expected value}
\label{App:ExpVal}

Let
$$F(\rho) = \E_\rho[f] = \int_{\R^n} \rho(x) f(x) \, dx$$
for some smooth function $f \colon \R^n \to \R$ in space.
If $f$ is minimized at $x^\ast(f) = \arg\min_{x \in \R^n} f(x)$ with minimum value $\min f$,
then $F$ is minimized at any measure $\rho$ supported on $x^\ast(f)$ with the same minimum value, $\min F = \min f$.
Essentially, the behavior of $F$ is the same as that of $f$.

\subsubsection{Gradient domination and strong convexity}

The gradient of $F(\rho) = \E_\rho[f]$ is $\grad_\rho F = -\nabla \cdot (\rho \nabla f) \equiv \nabla f$, so 
$\|\grad_\rho F\|^2 = \E_\rho[\|\nabla f\|^2]$.

The gradient-domination condition $\|\grad F\|^2 \ge 2\alpha (F-\min F)$, $\alpha > 0$, becomes
$$\E_\rho[\|\nabla f\|^2] \ge 2\alpha \, \E_\rho[f-\min f]~~~~\forall \, \rho \in \P.$$
Therefore, $F$ satisfies gradient domination in the space of measures if and only if $f$ satisfies the gradient-domination condition $\|\nabla f\|^2 \ge 2\alpha(f - \min f)$ in space.

Similarly, the Hessian of $F$ is given by $(\Hess_\rho F)(R,R) = \E_\rho[\langle \nabla \phi, (\nabla^2 f) \nabla \phi \rangle]$, for a tangent function $R \equiv \nabla \phi \in \T_\rho \P$.
Therefore,
$F$ is strongly convex in the space of measures if and only if $f$ is strongly convex in space.
In particular, if $f$ (and hence $F$) is strongly convex, then $f$ has a unique minimizer $x^\ast \in \R^n$, and $F$ has a unique minimizer $\delta_{x^\ast} \in \P$ which is the point mass at $x^\ast$.

\subsubsection{Gradient flow}
\label{Sec:GFF}

The gradient flow equation $\dot \rho = -\grad_\rho F$ is
$$\part{\rho}{t} = \nabla \cdot \left(\rho \nabla f \right).$$
This is  the continuity equation of the gradient flow equation of $f$ in space:
$$\dot x = -\nabla f(x).$$
Therefore, the gradient flow of $F(\rho) = \E_\rho[f]$ 
is implementable by the gradient flow of $f$ in space.
Furthermore, the rate of convergence for $F$ is the same as the rate of convergence for $f$.
If $f$ (and hence $F$) is gradient dominated, then the convergence is exponential in the function value, and also in the distance to the set of minimizers.
If $f$ (and hence $F$) is strongly convex, then any two coevolving solutions are contracting exponentially fast.
If $f$ (and hence $F$) is nonconvex, 
then the limiting measure is supported on the set of local minima, with weights equal to the measure of the corresponding basins of attraction under the initial measure.

\subsubsection{Gradient descent}
\label{Sec:GDF}

The forward method or gradient descent $\rho_{k+1} = \exp_{\rho_k}(-\epsilon \grad_{\rho_k} F)$ becomes
\begin{align}\label{Eq:GDMeas}
\rho_{k+1} = \exp_{\rho_k}(-\epsilon \nabla f)
\end{align}
where recall $\grad_\rho F \equiv \nabla f$.
Assume $f$ is $L$-smooth, which means $\nabla^2 f \preceq L I$.
Then for $\epsilon < \frac{1}{L}$, the gradient descent for $F$~\eqref{Eq:GDMeas} is implemented by gradient descent for $f$ in space:
$$\rho_{k+1} = (I - \epsilon \nabla f)_\# \rho_k.$$
Furthermore, if $f$ (and hence $F$) is $\alpha$-gradient-dominated for some $\alpha > 0$, then gradient descent converges exponentially fast in the function value:
$$F(\rho_k) - \min F \le (1-\alpha \epsilon (2-\epsilon L))^k (F(\rho_0)-\min F).$$
For example, for $\epsilon = \frac{1}{2L}$, the rate is $(1-\frac{3}{4}\frac{\alpha}{L})^k$.
This also implies an exponential convergence rate in the distance to minimizer.
Similarly, if $f$ (and hence $F$) is strongly convex, then any two solutions are contracting exponentially fast.

\subsubsection{Proximal gradient}
\label{Sec:PGF}

The backward method or proximal gradient algorithm for $F(\rho) = \E_\rho[f]$ is
$$\rho_{k+1} = \arg\min_{\rho \in \P} \left\{ \E_\rho[f] + \frac{1}{2\epsilon} W_2(\rho,\rho_k)^2 \right\}.$$
Equivalently, $\rho_{k+1}$ solves the adjoint version of the equation~\eqref{Eq:GDMeas}:
$$\exp_{\rho_{k+1}}(\epsilon \nabla f) = \rho_k.$$
This is implemented by the backward method or proximal gradient algorithm for $f$ in space:
$$\rho_{k+1} = \left((I+\epsilon \nabla f)^{-1}\right)_\# \rho_k.$$
See below for detail.
Here $(I+\epsilon \nabla f)^{-1}$ is the proximal gradient operator for $f$; i.e., $x_{k+1} = (I+\epsilon \nabla f)^{-1}(x_k)$ if and only if
$$x_{k+1} = \arg\min_{x \in \R^n} \left\{ f(x) + \frac{1}{2\epsilon} \|x-x_k\|^2 \right\}.$$
This is well-defined if $f$ satisfies $-LI \preceq \nabla^2 f \preceq LI$ for some $L > 0$, and $0 < \epsilon < \frac{1}{L}$.
Furthermore, if $f$ (and hence $F$) is $\alpha$-gradient-dominated for some $\alpha > 0$, then proximal gradient converges exponentially fast in the function value:
$$F(\rho_k) - \min F \le \frac{1}{(1+\alpha \epsilon)^k} (F(\rho_0)-\min F).$$
This also implies an exponential convergence rate in the distance to minimizer.
Similarly, if $f$ (and hence $F$) is strongly convex, then any two solutions are contracting exponentially fast.

\subsubsection{Detail for~$\S\ref{Sec:PGF}$: Proximal gradient for expected value}
\label{App:PGF}

Let $f \colon \R^n \to \R$ be a smooth function.

The {\em proximal step} for $f$ with step size $\epsilon > 0$ is the map $P_{\epsilon,f} \colon \R^n \to \R^n$ given by
$$P_{\epsilon,f}(x) = \arg\min_{y \in \R^n} \left\{ f(y) + \frac{1}{2\epsilon} \|y-x\|^2 \right\}.$$
Taking derivative and setting it to zero, the solution $y = P_{\epsilon,f}(x)$ satisfies
$\nabla f(y) + \frac{1}{\epsilon}(y-x) = 0$, or equivalently,
$y + \epsilon \nabla f(y) = x.$
Thus,
$$P_{\epsilon,f}(x) = (I + \epsilon \nabla f)^{-1}(x).$$
This is well-defined if $f$ is $L$-smooth, i.e., $-LI \preceq \nabla^2 f(x) \preceq LI$, and $0 < \epsilon < \frac{1}{L}$.

In the space of measure, this corresponds to the pushforward map
$$\widetilde P_{\epsilon,f}(\nu) = ((I+\epsilon \nabla f)^{-1})_\# \nu.$$
We will show this is also the proximal step in the space of measure of the corresponding expected value functional:
$$\widetilde P_{\epsilon,f}(\nu) = \arg\min_{\rho \in \P} \left\{ \E_\rho[f] + \frac{1}{2\epsilon} W(\rho,\nu)^2\right\}.$$

We calculate the derivative and set it to zero at stationary point. 
At the minimizer $\rho = \widetilde P_{\epsilon,f}(\nu)$, the gradient of $\widetilde F(\rho) = \E_\rho[f] + \frac{1}{2\epsilon} W(\rho,\nu)^2$ vanishes.
This means for any motion of $\rho$ given by $\part{\rho}{t} + \nabla \cdot (\rho \xi) = 0$, we have
$$0 = \left.\frac{d}{dt} \right|_{t=0} \widetilde F(\rho) = \int \rho(y) \left\langle \nabla f(y) + \frac{1}{\epsilon}(y-\nabla \varphi^\ast(y)), \, \xi(y) \right \rangle \, dy.$$
Here $\nabla \varphi$ is the optimal transport map that sends $\nu$ to $\rho$, so $I - \nabla \varphi^\ast$ is the optimal map that sends $\rho$ to $\nu$.
Since the integral above is zero for any $\xi$, we must have ($\rho$--almost surely)
$$y + \epsilon \nabla f(y) = \nabla \varphi^\ast(y).$$
Thus, 
$$\nabla \varphi^\ast = I + \epsilon \nabla f$$
so the optimal map that sends $\nu$ to $\rho$ is
$$\nabla \varphi = (I + \epsilon \nabla f)^{-1}$$
which is the proximal step map.
Therefore, the solution $\rho = \widetilde P_{\epsilon,f}(\nu)$ is given by
$$\rho = (\nabla \varphi)_\# \nu = ((I+\epsilon \nabla f)^{-1})_\# \nu$$
which is the same as the pushforward of the proximal map above, as desired.

\subsection{Negative entropy}
\label{App:Ent}

Suppose
$$F(\rho) = -H(\rho) = \int_{\R^n} \rho(x) \log \rho(x) \, dx$$
is the negative entropy.
It takes values in $\R \cup \{+\infty\}$.
The value $+\infty$ is achieved by a point mass, or any distribution supported on a set of Lebesgue measure zero.
Negative entropy is also unbounded below, with the limiting value $-\infty$ achieved by the Lebesgue measure (which is not a probability measure).
For example, for a Gaussian distribution $\rho = \N(\mu,\Sigma)$, the negative entropy is $F(\rho) = -\frac{1}{2} \log \det(2 \pi e \Sigma)$.

\subsubsection{The heat flow as gradient flow of negative entropy}
\label{Sec:Heat}

The gradient of $F(\rho) = -H(\rho)$ is $\grad_\rho F = -\nabla \cdot (\rho \nabla \log \rho) = -\Delta \rho$,
or $\grad_\rho F \equiv \nabla \log \rho$.
Therefore, the gradient flow equation $\dot \rho = -\grad_\rho F$ of negative entropy is the {\em heat equation}:
$$\part{\rho}{t} = \Delta \rho.$$
This has an exact solution which is the {\em heat flow}:
$$\rho_t = \rho_0 \ast \N(0,2t I).$$
In space, this is implemented via addition of independent Gaussian noise:
\begin{align}\label{Eq:XZ}
X_t = X_0 + \sqrt{2t} Z
\end{align}
where $Z \sim \N(0,I)$ is independent of $X_0$.\footnote{The true solution of the heat flow is the Brownian motion in space.
However, at each time, the solution has the same distribution as~\eqref{Eq:XZ} above.}
Note that $\rho_t$ tends to the Lebesgue measure as $t \to \infty$.
Note also that variance grows linearly along the heat flow.

For example, if $\rho_0 = \N(\mu,\Sigma)$, then $\rho_t = \N(\mu,\Sigma + 2tI)$.
In this case negative entropy decreases logarithmically:
$F(\rho_t) = -\frac{n}{2} \log (2\pi e) - \frac{1}{2} \sum_{i=1}^n \log (\lambda_i + 2t)$, where $\lambda_1,\dots,\lambda_n > 0$ are the eigenvalues of $\Sigma \succ 0$.
This logarithmic rate is also the fastest rate at which negative entropy decreases along the heat flow.
This follows from the fact that 
Gaussian minimizes negative entropy for a given covariance, so
$F(\rho_t) \ge F(\N(0,\Sigma+2tI))$ where $\Sigma$ is the covariance of $\rho_0$.

By the nonnegativity of mutual information, we also have a logarithmic upper bound on negative entropy.
Indeed, since $I(X_0;X_t) = H(X_t) - H(X_t \,|\, X_0) \ge 0$ and $H(X_t\,|\,X_0) = H(\N(0,2tI))$, we have $F(\rho_t) \le F(\N(0,2tI))$.
However, the upper bound is $+\infty$ at $t = 0$.
Using the entropy power inequality, 
we can improve this to $F(\rho_t) \le -\frac{n}{2} \log(e^{-\frac{2}{n} F(\rho_0)} + e^{-\frac{2}{n} F(\N(0,2tI))})$,
which has the same logarithmic behavior for large $t$, and takes the correct value $F(\rho_0)$ at $t = 0$.

\subsubsection{Fisher information and convexity of negative entropy}
\label{Sec:Fish}

The squared norm of the gradient of negative entropy $F = -H$ is the {\em Fisher information}:
$$J(\rho) = \E_\rho[\|\nabla \log \rho\|^2].$$
Therefore, the gradient flow identity $\frac{d}{dt} F(\rho) = -\|\grad_\rho F\|^2$ becomes $\frac{d}{dt} H(\rho) = J(\rho)$, which is known as De Bruijn's identity in information theory~\cite[]{Sta59}.

The Hessian of $F(\rho) = -H(\rho)$ is, on a tangent function $R = -\nabla \cdot(\rho \nabla \phi) \in \T_\rho \P$,
$$(\Hess_\rho F)(R,R) = \E_\rho[\|\nabla^2 \phi\|^2_{\HS}].$$
The expression above is nonnegative; therefore, negative entropy is a convex functional in the space of measures.
It is not strictly convex in general,
but it is strongly convex along geodesics that preserve the mean, with strong convexity parameter the Poincar\'e constant of the distribution; see~\cite[Corollary~2.5]{CG03}.
Note that the Hessian of negative entropy is not bounded above.

We also note that the gradient flow identity $\frac{d^2}{dt^2} F(\rho) = 2(\Hess F)(\grad F, \grad F)$ becomes $\frac{d^2}{dt^2} H(\rho) = -2K(\rho)$, where $K(\rho)$ is the {\em second-order Fisher information}:
$$K(\rho) = \E_\rho[\|\nabla^2 \log \rho\|^2_{\HS}].$$
This is used in the proof of the concavity of entropy power along the heat flow~\cite[]{Dembo91,Vil00}.
Moreover, since $\frac{d}{dt} H(\rho) = J(\rho) \ge 0$ and $\frac{d^2}{dt^2} H(\rho) = -2K(\rho) \le 0$,
entropy is increasing and concave along the heat flow;
this is also a consequence of the general fact that a convex function is decreasing in a convex manner along its own gradient flow.

\subsubsection{Forward method for heat flow}
\label{Sec:FwHeat}

The forward method $\rho_{k+1} = \exp_{\rho_k}(-\epsilon \grad_{\rho_k} F)$ for negative entropy $F = -H$ is
$$\rho_{k+1} = \exp_{\rho_k}(-\epsilon \nabla \log \rho_k)$$
where recall $\grad_\rho F \equiv \nabla \log \rho$.
If $\rho_k$ is $K$-log-semiconcave, which means $-\nabla^2 \log \rho_k \succeq KI$ for some $K \in \R$, then for $\epsilon \le \frac{1}{\max\{0,-K\}}$, the forward method above is implemented by:
\begin{align}\label{Eq:FwHeat}
\rho_{k+1} = (I - \epsilon \nabla \log \rho_k)_\# \rho_k.
\end{align}
If we know the analytic form of $\nabla \log \rho_k$, then we can run one step of the algorithm above by $x_{k+1} = x_k - \epsilon \nabla \log \rho_k(x_k)$.
However, in the next step we do not know the analytic form of $\nabla \log \rho_{k+1}$, so we cannot continue running it.

For Gaussian initial data, we can solve the forward method.
In this case the distributions stay Gaussian, but the variance grows faster than in the heat flow.

\begin{example}[Forward method for heat flow with Gaussian data.]
Let $\rho_0 = \N(\mu_0,\Sigma_0)$, $\epsilon > 0$.
Along the forward method~\eqref{Eq:FwHeat} for the heat flow, $\rho_k = \N(\mu_k,\Sigma_k)$ stays Gaussian.
Furthermore, the algorithm~\eqref{Eq:FwHeat} becomes the iteration
$x_{k+1} = (I + \epsilon \Sigma_k^{-1}) x_k - \epsilon \Sigma_k^{-1} \mu_k.$
Since $x_k \sim \N(\mu_k,\Sigma_k)$, $x_{k+1} \sim \N(\mu_k, \Sigma_k(I + \epsilon \Sigma_k^{-1})^2).$
Therefore, $\mu_{k+1} = \mu_k = \mu_0$ and $\Sigma_{k+1} = \Sigma_k(I + \epsilon \Sigma_k^{-1})^2$. 
Note that $\Sigma_k \succ \Sigma_{k-1} + 2\epsilon I \succ \Sigma_0 + 2\epsilon kI$, so 
variance grows faster than in the heat flow (with $t = \epsilon k$).
\end{example}

\subsubsection{Backward method for heat flow}
\label{Sec:BwHeat}

The backward method $\rho_{k+1} = \arg\min_{\rho \in \P} \{F(\rho) + \frac{1}{2\epsilon} W_2(\rho,\rho_k)^2\}$ for negative entropy  $F = -H$ is
$$\exp_{\rho_{k+1}}(\epsilon \nabla \log \rho_{k+1}) = \rho_k.$$
If $\rho_{k+1}$ is $L$-log-smooth, which means $-\nabla^2 \log \rho_{k+1} \preceq LI$ for some $L > 0$, then for $\epsilon \le \frac{1}{L}$, the backward method above is implemented (implicitly) by:
\begin{align}\label{Eq:BwHeat}
(I + \epsilon \nabla \log \rho_{k+1})_\# \rho_{k+1} = \rho_k.
\end{align}
In general this is not solvable analytically.
For Gaussian initial data, we can solve the backward method.
In this case the distributions stay Gaussian, but the variance grows slower than in the heat flow (and in contrast to the forward method), see also~\cite[Remark, $\S3$]{CG03}.

\begin{example}[Backward method for heat flow with Gaussian data.]
Let $\rho_0 = \N(\mu_0,\Sigma_0)$ and $0 < \epsilon \le \lambda_{\min}(\Sigma_0)$.
Along the backward method~\eqref{Eq:BwHeat} for the heat flow, $\rho_k = \N(\mu_k,\Sigma_k)$ stays Gaussian.
Furthermore, the algorithm~\eqref{Eq:BwHeat} becomes the iteration
$(I - \epsilon \Sigma_{k+1}^{-1}) x_{k+1} + \epsilon \Sigma_{k+1}^{-1} \mu_{k+1} = x_k$.
Since $x_{k+1} \sim \N(\mu_{k+1},\Sigma_{k+1})$, $x_k \sim \N(\mu_{k+1}, \Sigma_{k+1}(I - \epsilon \Sigma_{k+1}^{-1})^2).$
Therefore, $\mu_{k+1} = \mu_k = \mu_0$ and $\Sigma_{k+1}(I - \epsilon \Sigma_{k+1}^{-1})^2 = \Sigma_k$.
Equivalently, $\Sigma_{k+1} = \frac{1}{2}(\Sigma_k+2\epsilon I + \{\Sigma_k(\Sigma_k+4\epsilon I)\}^{\frac{1}{2}})
= \Sigma_k + 2\epsilon I - \epsilon^2 \Sigma_k^{-1} + O(\epsilon^3)$.
In particular, $\Sigma_k \prec \Sigma_{k-1} + 2\epsilon I \prec \Sigma_0 + 2\epsilon k I$, so 
variance grows slower than in the heat flow. 
\end{example}

\subsection{Gradient flow of variance}
\label{App:Var}

The variance of a probability density $\rho$ on $\R^n$ can be written as
$$\Var(\rho) = \frac{1}{2} \int_{\R^n \times \R^n} \rho(x) \rho(y) \|x-y\|^2 \, dx \, dy.$$
This is an example of an ``interaction energy'' of~\cite[$\S5.2.2$]{Vil03}, which in general is of the form $F(\rho) = \frac{1}{2} \int_{\R^n \times \R^n} \rho(x) \rho(y) W(x-y) \, dx \, dy$ where $W \colon \R^n \to \R$ is a symmetric convex function.
Variance is the case when $W(x) = \|x\|^2$.

In general, the gradient of the interaction energy is $\grad_\rho F = -\nabla \cdot(\rho \nabla (\rho \ast W)) \equiv \nabla (\rho \ast W)$, where $\ast$ is the convolution.
For variance with $W(x) = \|x\|^2$, we have $(\rho \ast W)(x) = \int_{\R^n} \rho(y) \|x-y\|^2 \, dy$. 
Then $\nabla (\rho \ast W)(x) = 2 \int_{\R^n} \rho(y) (x-y) \, dy = 2(x-\mu(\rho))$ where $\mu(\rho) = \E_\rho[X]$ is the mean of $\rho$.
Therefore, the gradient of variance is
$\grad_\rho \Var = -\nabla \cdot (\rho \nabla (\rho \ast W)) = -2\nabla \cdot (\rho (x-\mu(\rho)))$.
Notice that
$$\|\grad_\rho \Var\|^2 = \int \rho \|\nabla (\rho \ast W)\|^2 = 4 \int \rho \|x-\mu(\rho)\|^2 = 4\Var(\rho).$$
Therefore, along the gradient flow of variance, $\frac{d}{dt} \Var(\rho) = -\|\grad_\rho \Var\|^2 = -4\Var(\rho)$, which implies $\Var(\rho_t) = e^{-4t} \Var(\rho_0)$.

The gradient flow of variance $\dot \rho = -\grad_\rho \Var$ is
$$\part{\rho}{t} = 2\nabla \cdot (\rho (x-\mu(\rho))).$$
Observe that the mean is preserved:
$$\frac{d}{dt} \mu(\rho) = \int_{\R^n} \part{\rho}{t} x = 2 \int_{\R^n} \nabla \cdot (\rho (x-\mu(\rho))) x = -2\int_{\R^n} \rho (x-\mu(\rho)) = -2(\mu(\rho)-\mu(\rho)) = 0$$
where in the calculation above we have used integration by parts.
Therefore, $\mu(\rho) = \mu(\rho_0) = \mu_0$ is fixed along the gradient flow of variance.
Thus, the gradient flow of variance becomes
$$\part{\rho}{t} = 2\nabla \cdot (\rho (x-\mu_0)).$$
This has an exact solution
$$\rho(t,x) = e^{2nt} \rho_0(e^{2t} x + (I-e^{2t})\mu_0).$$
Furthermore, this is implemented in space by shrinking around the mean:
$$x(t) = \mu_0 + e^{-2t}(x-\mu_0).$$

\section{A review of composite optimization}
\label{App:Comp}

Let $\M$ be an $n$-dimensional smooth Riemannian manifold.
We consider the composite optimization problem
$$\min_{x \in \M} f(x) + g(x)$$
where $f,g \colon \M \to \R$ are smooth functions.
One way to solve this optimization problem is to apply basic algorithms---gradient flow, gradient descent, or proximal gradient---to the objective function $f + g$.
However, sometimes we can compute these algorithms on $f$ and $g$ individually, but not on their sum; an example is the Langevin dynamics in discrete time.

One algorithm for solving this optimization problem is the {\em Forward-Backward} (FB) algorithm, which alternates between a forward step (gradient descent) for $f$ and a backward step (proximal gradient) for $g$.
The iteration can be written as a composition of two steps, with a step size $\epsilon > 0$:
\begin{subequations}\label{Eq:FB}
\begin{align}
x_{k+\frac{1}{2}} &= \exp_{x_k}(-\epsilon \, \grad_{x_k} f) \label{Eq:FBa} \\ 
x_{k+1} &= \arg\min_{x\in\M} \left\{ g(x) + \frac{1}{2\epsilon} d(x,x_{k+\frac{1}{2}})^2\right\}. \label{Eq:FBb}
\end{align}
\end{subequations}
Recall that the optimality condition for the second equation~\eqref{Eq:FBb} is
$\exp_{x_{k+1}}(\epsilon \, \grad_{x_{k+1}} g ) = x_{k+\frac{1}{2}}.$
Combining this with~\eqref{Eq:FBa}, we can write the FB algorithm as defined implicitly by the identity
\begin{align}\label{Eq:FBid}
\exp_{x_{k+1}}(\epsilon \, \grad_{x_{k+1}} g ) = \exp_{x_k}(-\epsilon \, \grad_{x_k} f).
\end{align}

\subsection{Why does Forward-Backward work?}
\label{App:WhyFB}

The minimizer (or any stationary point) of $f+g$ is a fixed point of the forward-backward algorithm.
This is because the backward step (proximal gradient) is the inverse of the negative forward step (gradient descent).

Concretely, let $x^\ast \in \M$ be a stationary point of $f+g$, so $\grad_{x^\ast} f + \grad_{x^\ast} g = 0$.
Suppose we start at $x_k = x^\ast$.
Then in the first half-step we move to $x_{k+\frac{1}{2}} = \exp_{x^\ast}(-\epsilon \grad_{x^\ast} f)$, which in general is different from $x^\ast$ (because $x^\ast$ may not be a stationary point of $f$, so $\grad_{x^\ast} f \neq 0$).
However, the next half-step brings us back to $x_{k+1} = x^\ast$; this is because $x_{k+1}$ by definition satisfies the consistency equation 
$$\exp_{x_{k+1}}(\epsilon \, \grad_{x_{k+1}} g) = x_{k+\frac{1}{2}} = \exp_{x^\ast}(-\epsilon \, \grad_{x^\ast} f)$$
and we see that $x_{k+1} = x^\ast$ is a solution since $\grad_{x^\ast} g = -\grad_{x^\ast} f$.

This pairing between methods which are inverses of each other is important to make the algorithm converge to the true minimizer of the composite function.
By symmetry, we can also use the backward and forward method, which will also converge to the correct minimizer.
However, if we choose any other pairing, then the algorithm will have a bias, i.e., it will converge to a point that is different from the true minimizer.

\subsubsection{Example: Quadratic function in $\R$}

Let $\M = \R$ and $f(x) = \frac{1}{2} (x-a)^2$, for some $a \in \R$.
The basic algorithms with step size $\epsilon > 0$ are:

\renewcommand{\arraystretch}{1.1}
\begin{table}[h!]
\begin{center}
    \begin{tabular}[h!t]{c | l}
    Algorithm & ~~~~~~~~~~~~ Iteration \\
    \hline
    Gradient descent (GD) & ~ $x_{k+1} = (1-\epsilon) x_k + \epsilon a$ \\
    Gradient flow (GF) & ~ $x_{k+1} = e^{-\epsilon} x_k + (1-e^{-\epsilon}) a$ \\
    Proximal gradient (PG) ~ & ~ $x_{k+1} = \frac{1}{1+\epsilon} x_k + \frac{\epsilon}{1+\epsilon} a$
    \end{tabular}
\caption{\footnotesize Basic algorithms applied to $f(x) = \frac{1}{2}(x-a)^2$ in $\R$.}
\label{Tab:quad}
\end{center}
\end{table}

\noindent
Observe that $1-\epsilon \le e^{-\epsilon} \le \frac{1}{1+\epsilon}$, so in the quadratic case we see that PG is faster than GF which in turn is faster than GD.

Now consider the composite optimization problem $\min_{x \in \R} f(x) + g(x)$ where
$$f(x) = \frac{1}{2}(x-1)^2 ~~~~~~ \text{ and } ~~~~~~ g(x) = \frac{1}{2} (x+1)^2.$$
Then $f(x)+g(x) = x^2 + 1$, which is minimized at $x^\ast = 0$; but note that $x^\ast = 0$ is not a stationary point of $f$ or $g$.

We consider composite algorithms to solve the composite optimization problem in which we alternately apply a basic algorithm to each of $f$ and $g$:
\begin{align*}
x_{k+\frac{1}{2}} &= \textrm{A}_{\epsilon,f}(x_k) \\
x_{k+1} &= \textrm{A}'_{\epsilon,g}(x_{k+\frac{1}{2}})
\end{align*}
for some $\textrm{A}, \textrm{A}' \in \{\textrm{GD}, \textrm{GF}, \textrm{PG}\}$.
For each combination, we use Table~\ref{Tab:quad} to compute the exact iteration from $x_k$ to $x_{k+1}$ and determine the limit point.
If the limit point is not $x^\ast = 0$, then the composite algorithm is biased; else, it is unbiased.
The results are in Table~\ref{Tab:quadcomp}.
Note that all are biased except for two: the Forward-Backward algorithm, and the Backward-Forward algorithm.

\renewcommand{\arraystretch}{1.5}
\begin{table}[h!]
\begin{center}
    \begin{tabular}[h!t]{c | c | l | c }
    Alg.~for $f$ & ~ Alg.~for $g$ & ~~~~~~~~~~ Iteration & ~ Limit \\
    \hline
    GD & GD & ~ $x_{k+1} = (1-\epsilon)^2 x_k - \epsilon^2$ ~ & $-\frac{\epsilon}{2-\epsilon}$ \\
     & GF & ~ $x_{k+1} = e^{-\epsilon}(1-\epsilon) x_k + e^{-\epsilon}(1+\epsilon)-1$ ~ & $\frac{e^{-\epsilon}(1+\epsilon)-1}{1-e^{-\epsilon}(1-\epsilon)}$   \\
     & PG & ~ $x_{k+1} = \frac{1-\epsilon}{1+\epsilon} x_k$ & 0  \\
     \hline
    GF & GD & ~ $x_{k+1} = e^{-\epsilon}(1-\epsilon) x_k + (1-e^{-\epsilon})(1-\epsilon)-\epsilon$  & $\frac{(1-e^{-\epsilon})(1-\epsilon)-\epsilon}{1-e^{-\epsilon}(1-\epsilon)}$   \\
     & GF & ~ $x_{k+1} = e^{-2\epsilon} x_k - (1-e^{-\epsilon})^2$ ~ & $-\frac{(1-e^{-\epsilon})^2}{1-e^{-2\epsilon}}$  \\
     & PG & ~ $x_{k+1} = \frac{e^{-\epsilon}}{1+\epsilon} x_k + \frac{1-e^{-\epsilon}-\epsilon}{1+\epsilon}$ ~ & $\frac{1-e^{-\epsilon}-\epsilon}{1-e^{-\epsilon}+\epsilon}$  \\
     \hline
    PG & GD & ~ $x_{k+1} = \frac{1-\epsilon}{1+\epsilon} x_k$ & 0  \\
     & GF & ~ $x_{k+1} = \frac{e^{-\epsilon}}{1+\epsilon} x_k + \frac{e^{-\epsilon}\epsilon}{1+\epsilon}-1 + e^{-\epsilon}$ & $\frac{e^{-\epsilon}\epsilon-(1-e^{-\epsilon})(1+\epsilon)}{1+\epsilon-e^{-\epsilon}}$  \\
     & PG & ~ $x_{k+1} = \frac{1}{(1+\epsilon)^2} x_k + \frac{\epsilon^2}{(1+\epsilon)^2}$ & $\frac{\epsilon}{2-\epsilon}$  \\
    \end{tabular}
\caption{\footnotesize Composite algorithms applied to $f(x) = \frac{1}{2}(x-1)^2$ and $g(x) = \frac{1}{2}(x+1)^2$ in $\R$.}
\label{Tab:quadcomp}
\end{center}
\end{table}

\subsection{Convergence rate of FB under gradient domination in Euclidean space}
\label{App:FBRn}

We review the exponential convergence rate of the Forward-Backward (FB) algorithm under gradient domination condition and partial smoothness assumptions.
In this section we are working in the Euclidean case $M = \R^n$.
The following is adapted from~\cite[Theorem~4.2.b.ii]{GRV17}.

\begin{lemma}\label{Lem:FBRn}
Let $f \colon \R^n \to \R$ be $K$-semiconvex and $L$-smooth ($KI \preceq \nabla^2 f \preceq LI$) for some $K \in \R$, $L > \max\{0,-K\}$, and let $g \colon \R^n \to \R$ be convex ($\nabla^2 g \succeq 0$).
Assume $\tilde f = f+g$ is $\alpha$-gradient dominated for some $\alpha > 0$.
Consider the FB iteration
\begin{align*}
x_{k+\frac{1}{2}} &= x_k - \epsilon \nabla f(x_k) \\
x_{k+1} &= \arg\min_{x \in \R^n} \left\{ g(x) + \frac{1}{2\epsilon} \|x-x_{k+\frac{1}{2}}\|^2\right\}.
\end{align*}
Then for $0 < \epsilon \le \min\{\frac{2}{L},\frac{2}{K+L}\}$,
$$\tilde f(x_k) - \min \tilde f \le \left(1 + \frac{\alpha \epsilon(2-\epsilon L)}{1-2\epsilon\frac{KL}{K+L}} \right)^{-k} (\tilde f(x_0) - \min \tilde f).$$
\end{lemma}
\begin{proof}
By the $L$-smoothness of $f$ and the convexity of $g$,
\begin{align*}
f(x_{k+1}) &\le f(x_k) + \langle \nabla f(x_k), x_{k+1}-x_k \rangle + \frac{L}{2} \|x_{k+1}-x_k\|^2 \\
g(x_{k+1}) &\le g(x_k) - \langle \nabla g(x_{k+1}), x_k-x_{k+1} \rangle.
\end{align*}
Adding the two inequalities above yields
$$\tilde f(x_{k+1}) \le \tilde f(x_k) + \langle \nabla f(x_k) + \nabla g(x_{k+1}), x_{k+1}-x_k \rangle + \frac{L}{2} \|x_{k+1}-x_k\|^2.$$
Since $x_k-\epsilon\nabla f(x_k) = x_{k+\frac{1}{2}} = x_{k+1} + \epsilon \nabla g(x_{k+1})$, we have $\nabla f(x_k) + \nabla g(x_{k+1}) = \frac{1}{\epsilon}(x_k-x_{k+1})$.
Note that here the linearity $\R^n$ is essential.
Substituting this equality to the inequality above yields
$$\tilde f(x_{k+1}) \le \tilde f(x_k) - \left(\frac{1}{\epsilon}-\frac{L}{2}\right) \|x_{k+1}-x_k\|^2.$$
Since $f$ is $K$-semiconvex and $L$-smooth, we have $\|x_{k+1}-x_k\|^2 \ge (1-2\epsilon\frac{KL}{K+L})^{-1} \|x_{k+\frac{3}{2}}-x_{k+\frac{1}{2}}\|^2$
by Lemma~\ref{Lem:SemiConvex} below, so
$$\tilde f(x_{k+1}) \le \tilde f(x_k) - \frac{1}{\left(1-2\epsilon\frac{KL}{K+L}\right)} \left(\frac{1}{\epsilon}-\frac{L}{2}\right) \|x_{k+\frac{3}{2}}-x_{k+\frac{1}{2}}\|^2.$$
Furthermore, 
$x_{k+\frac{3}{2}}-x_{k+\frac{1}{2}} = \left( x_{k+1}-\epsilon \nabla f(x_{k+1}) \right) - \left( x_{k+1}+\epsilon \nabla g(x_{k+1}) \right)
= -\epsilon (\nabla f(x_{k+1}) + \nabla g(x_{k+1})) = -\epsilon \nabla \tilde f(x_{k+1}).$
Note that here the linearity of $\R^n$ is also essential.
Therefore,
$$\tilde f(x_{k+1}) \le \tilde f(x_k) - \frac{\epsilon^2}{\left(1-2\epsilon\frac{KL}{K+L}\right)} \left(\frac{1}{\epsilon}-\frac{L}{2}\right) \|\nabla \tilde f(x_{k+1})\|^2.$$
Now using the $\alpha$-gradient domination assumption on $\tilde f$, we get
$$\tilde f(x_{k+1}) \le \tilde f(x_k) - \frac{2\alpha\epsilon^2}{\left(1-2\epsilon\frac{KL}{K+L}\right)} \left(\frac{1}{\epsilon}-\frac{L}{2}\right) (\tilde f(x_{k+1})-\min \tilde f).$$
Collecting terms and unrolling the recursion give
$$\tilde f(x_k) - \min \tilde f \le \left(1 + \frac{\alpha \epsilon(2-\epsilon L)}{1-2\epsilon\frac{KL}{K+L}} \right)^{-k} (\tilde f(x_0) - \min \tilde f)$$
as desired.
\end{proof}

The proof of Lemma~\ref{Lem:FBRn} uses the following lemma.

\begin{lemma}\label{Lem:SemiConvex}
Let $f \colon \R^n \to \R$ be $K$-semiconvex and $L$-smooth ($KI \preceq \nabla^2 f \preceq LI$) for some $K \in \R$, $L > \max\{0,-K\}$.
Let $x,y \in \R^n$, and for $\epsilon \ge 0$, let $x(\epsilon) = x - \epsilon \nabla f(x)$ and $y(\epsilon) = y-\epsilon \nabla f(y)$.
Then for $0 \le \epsilon \le \frac{2}{K+L}$,
$$\|x(\epsilon)-y(\epsilon)\|^2 \le \left(1-2\epsilon\frac{KL}{K+L}\right) \|x-y\|^2.$$ 
\end{lemma}
\begin{proof}
Since $f$ is $K$-semiconvex and $L$-smooth, the function $\phi(x) = f(x) - \frac{K}{2}\|x\|^2$ is convex and $(L-K)$-smooth.
If $K = L$, then we are done.
Else, by a standard property of smooth convex function~\cite[Theorem~2.1.5]{Nesterov04} we have
$\langle \nabla \phi(x) - \nabla \phi(y), x-y \rangle \ge \frac{1}{L-K} \|\nabla \phi(x)-\nabla \phi(y)\|^2$.
Since $\nabla \phi(x) = \nabla f(x) - Kx$, this is equivalent to
$$\langle \nabla f(x)-\nabla f(y),x-y\rangle \ge \frac{KL}{K+L}\|x-y\|^2 + \frac{1}{K+L} \|\nabla f(x)-\nabla f(y)\|^2.$$
Note this is the same as the bound for a smooth and strongly convex function~\cite[Theorem~2.1.12]{Nesterov04}, but here $K$ may be negative (but not too negative, as $K+L>0$).
Then:
\begin{align*}
\|x(\epsilon)-y(\epsilon)\|^2 &= \|x-y-\epsilon(\nabla f(x)-\nabla f(y))\|^2 \\
&= \|x-y\|^2 - 2\epsilon \langle \nabla f(x)-\nabla f(y),x-y\rangle + \epsilon^2 \|\nabla f(x)-\nabla f(y)\|^2 \\
&\le \left(1-2\epsilon\frac{KL}{K+L}\right) \|x-y\|^2 - \epsilon \left(\frac{2}{K+L} - \epsilon\right)\|\nabla f(x)-\nabla f(y)\|^2.
\end{align*}
If $0 \le \epsilon \le \frac{2}{K+L}$, then the last term on the right hand side above is nonpositive, so we may drop it to get the desired result
$\|x(\epsilon)-y(\epsilon)\|^2 \le \left(1-2\epsilon\frac{KL}{K+L}\right) \|x-y\|^2.$
\end{proof}

\section{Details for $\S\ref{Sec:Comp}$}

\subsection{Forward-Backward for Langevin dynamics}
\label{App:FBLang}

\begin{example}[FB for OU with Gaussian data.]\label{Ex:OUFB}
Let $\nu = \N(\mu,\Sigma)$ as in Example~\ref{Ex:OU}, and $\epsilon \le \lambda_{\min}(\Sigma)$.
Let $\rho_0 = \N(\mu_0,\Sigma_0)$ and let $\Sigma_0 = I$ for simplicity, so $\Sigma_0$ commutes with $\Sigma$.
Along FB~\eqref{Eq:FBLang} for OU, $\rho_k = \N(\mu_k,\Sigma_k)$ stays Gaussian, and $\Sigma_k$ commutes with $\Sigma$.
Furthermore,~\eqref{Eq:FBLang} becomes: 
\begin{align*}
x_{k+\frac{1}{2}} &= \mu + (I-\epsilon \Sigma^{-1}) (x_k-\mu) \\
x_{k+1} &= \mu_{k+1} + (I - \epsilon \Sigma_{k+1}^{-1})^{-1}(x_{k+\frac{1}{2}}-\mu_{k+1}).
\end{align*}
Since $x_k \sim \N(\mu_k,\Sigma_k)$, this yields a system of updates (the second one is implicit):
\begin{align*}
\mu_{k+1} &= \mu + (I-\epsilon \Sigma^{-1})(\mu_k-\mu) \\
\Sigma_{k+1} (I-\epsilon \Sigma_{k+1}^{-1})^2 &= \Sigma_k(I-\epsilon \Sigma^{-1})^2.
\end{align*}
The mean converges exponentially fast to the correct mean:
$\mu_k = \mu + (I-\epsilon \Sigma^{-1})^k (\mu_0-\mu) \to \mu.$
For the covariance, note that $\Sigma_k = \Sigma_{k+1} = \Sigma$ is the only fixed point of the update.
Therefore, the FB algorithm is consistent for OU with Gaussian data.
\end{example}

\subsection{Backward-Forward for Langevin dynamics}
\label{App:BFLang}

\begin{example}[BF for OU with Gaussian data.]\label{Ex:OUBF}
Let $\nu = \N(\mu,\Sigma)$ as in Example~\ref{Ex:OU}, and $\epsilon \le \lambda_{\min}(\Sigma)$.
Let $\rho_0 = \N(\mu_0,\Sigma_0)$ and let $\Sigma_0 = I$ for simplicity, so $\Sigma_0$ commutes with $\Sigma$.
Along BF~\eqref{Eq:BFLang} for OU, $\rho_k = \N(\mu_k,\Sigma_k)$ stays Gaussian, and $\Sigma_k$ commutes with $\Sigma$.
Furthermore,~\eqref{Eq:BFLang} becomes: 
\begin{align*}
x_{k+\frac{1}{2}} &= \mu + (I+\epsilon \Sigma^{-1})^{-1} (x_k-\mu) \\
x_{k+1} &= \mu_k + (I + \epsilon \Sigma_k^{-1})(\mu-\mu_k) + (I + \epsilon \Sigma_k^{-1})(I-\epsilon \Sigma^{-1})^{-1} (x_k-\mu).
\end{align*}
Since $x_k \sim \N(\mu_k,\Sigma_k)$, this yields a system of updates
\begin{align*}
\mu_{k+1} &= \mu + (I+\epsilon \Sigma^{-1})^{-1}(\mu_k-\mu) \\
\Sigma_{k+1} &= \Sigma_k(I + \epsilon \Sigma_k^{-1})^2(I-\epsilon \Sigma^{-1})^{-2}.
\end{align*}
The mean converges exponentially fast to the correct mean:
$\mu_k = \mu + (I+\epsilon \Sigma^{-1})^{-k} (\mu_0-\mu) \to \mu.$
For the covariance, note that $\Sigma_k = \Sigma_{k+1} = \Sigma$ is the only fixed point of the update.
Therefore, the BF algorithm is consistent for OU with Gaussian data.
\end{example}

\end{document}